\documentclass[a4paper,11pt]{amsart}

\usepackage{graphicx}
\usepackage{amssymb}
\usepackage{latexsym}
\usepackage{amsfonts}
\usepackage{amsmath}
\usepackage{amsthm}
\usepackage{amssymb}
\usepackage{mathrsfs}
\usepackage{tabls}
\newtheorem{theorem}{Theorem}[section]
\newtheorem{lemma}[theorem]{Lemma}
\newtheorem{proposition}[theorem]{Proposition}
\newtheorem{corollary}[theorem]{Corollary}

\newtheorem{definition}[theorem]{Definition}
\newtheorem{example}[theorem]{Example}

\newtheorem{remark}[theorem]{Remark}

\newcommand{\C}{\mbb{C}}

\newcommand{\I}{\mc{I}}

\newcommand{\N}{\mbb{N}}

\newcommand{\R}{\mbb{R}}

\newcommand{\Q}{\mc{Q}}

\newcommand{\OO}{\Omega}

\newcommand{\mbb}{\mathbb}
\newcommand{\mc}{\mathcal}

\newcommand{\mk}{\mathfrak}
\newcommand{\mr}{\mathrm}

\newcommand{\mscr}{\mathscr}
\newcommand{\lra}{\longrightarrow}
\newcommand{\pr}{\prime}

\newcommand{\ui}{i}

\def\SS{\mathbb S}
\newcommand\Ac{A_{\C}}

\newcommand\Sl{\mathcal S}

\def\dd#1#2{\dfrac{\partial#1}{\partial#2}}

\newcommand\im{\operatorname{Im}}
\newcommand\re{\operatorname{Re}}

\newcommand{\cS}{\mbb{S}}

\newcommand{\q}{\mbb{H}}
\newcommand{\oc}{\mbb{O}}

\newcommand{\sto}{\mr{u}}
\newcommand{\Sto}{\mr{U}}
\newcommand{\punto}{\boldsymbol{\cdot}}%{\bullet}
\newcommand{\cl}{\mr{clos}}
\newcommand{\stx}{\mscr{S}}

\begin{document}

\title[Power and spherical series over real alternative $^{\ast}$-algebras]{Power and spherical series\\ over real alternative $^{\boldsymbol{\ast}}$-algebras} 

\author{Riccardo Ghiloni}
\email{ghiloni@science.unitn.it}

\author{Alessandro Perotti}
\email{perotti@science.unitn.it}
\thanks{Work partially supported by GNSAGA of INdAM, MIUR-PRIN project ``Variet\`a reali e complesse: geometria, topologia e analisi armonica" and MIUR-FIRB project ``Geometria Differenziale e Teoria Geometrica delle Funzioni"}
\address{Department of Mathematics, University of Trento, I--38123, Povo-Trento, Italy}


\maketitle

%%%%%%%%%%%%%%%%%%%%%%%%%%%%

\begin{abstract}
We study two types of series over  a real alternative $^*$-algebra $A$. The first type are series of the form $\sum_{n} (x-y)^{\punto n}a_n$, where $a_n$ and $y$ belong to $A$ and $(x-y)^{\punto n}$ denotes the $n$--th power of $x-y$ w.r.t.\ the usual product obtained by requiring commutativity of the indeterminate $x$ with the elements of $A$.  
In the real and in the complex cases, the sums of power series define, respectively, the real analytic and the holomorphic functions. 
In the quaternionic case, a series of this type produces, in the interior of its set of convergence, a function belonging to the recently introduced class of  slice regular functions. We show that also in the general setting of an alternative algebra $A$, the sum of a power series is a slice regular function. We consider also a second type of series, the spherical series, where the powers are replaced by a different sequence of slice regular polynomials. It is known that on the quaternions,  the set of convergence of these series is an open set, a property not always valid in the case of power series.  We characterize the sets of convergence of this type of series for an arbitrary alternative $^*$-algebra $A$. In particular, we prove that these sets are always open in the quadratic cone of $A$. Moreover, we show that every slice regular function has a spherical series expansion at every point.
\end{abstract}

%%%%%%%%%%%%%%%%%%%%%

\section{Introduction}

In a non--commutative setting, the ring of polynomials is usually defined by fixing the position of the coefficients w.r.t.\  the indeterminate $x$ (for example on the right) and by imposing commutativity of $x$ with the coefficients when two polynomials are multiplied together. In this way, one recovers some relevant results valid in the commutative case, and new phenomena appear (see e.g.~\cite[Sect.~16]{Lam}). We are interested in polynomials over  a real alternative algebra $A$, of the form $p(x)=\sum_{n=0}^d x^na_n$, with $a_n\in A$, or, more generally, of the type
\[
p(x)=\sum_{n=0}^d (x-y)^{\punto n}a_n
\]
where $a_n\in A$, $y\in A$ and $(x-y)^{\punto n}$ denotes the $n$--th power of $x-y$ w.r.t.\ the product described above, with the indeterminate $x$ commuting with the elements of $A$. The natural generalization of polynomials are power series of the form
\begin{equation}\label{series}
f(x)=\sum_{n=0}^{+\infty} (x-y)^{\punto n}a_n.
\end{equation}
When $A$ is $\R$ or $\C$, power series produce respectively real analytic or holomorphic functions on their disk of convergence. Then two questions immediately arise: 
\medskip

\noindent
\textsc{Question:} \emph{Which class of functions is obtained from power series of type \eqref{series}? And where these series converge?}
\medskip

In the quaternionic case, the answer to these questions was given in \cite{GeSto2012MathAnn}: the sum of a series of type \eqref{series}, in the interior of its set of convergence, is a \emph{slice regular} function. The theory of slice regularity on the quaternionic space was introduced by Gentili and Struppa in \cite{GeSt2006CR,GeSt2007Adv,GeStoSt2013}, and then it was generalized to Clifford algebras and alternative $^*$-algebras in \cite{CoSaSt2009Israel,  GhPe_Trends,GhPe_AIM}.

One critical aspect about quaternionic power series is that the set of convergence can have an empty interior. As proved in \cite{GeSto2012MathAnn}, if $y\notin\R$, this set may reduce to a disk centered at $y$ contained in the complex ``slice'' of the quaternionic space $\q$ spanned by the reals and by  $y$. To avoid this difficulty, a new series expansion, called spherical series, was introduced in \cite{StoppatoAdvMath2012}, where the powers $(x-y)^{\punto n}$ were replaced by another family of slice regular polynomials of a quaternionic variable. The set of convergence of these series, is always an open subset of $\q$ and, moreover, every slice regular function has a series expansion of this type near every point of its domain of definition.

In this paper we are able to answer to the aforementioned questions (also for spherical series) when $A$ is a real alternative $^*$-algebra, the more general setting where,  using the approach of \cite{GhPe_Trends,GhPe_AIM}, the concept of slice regularity can be defined.

\textit{Fix a real alternative algebra $A$ with unity $1$ of finite dimension, equipped with a real linear anti--involution $\mk{a}:A \lra A$}. We can then consider $A$ as a real $^*$-algebra. For simplicity, given any element $x$ of $A$, we will use the symbol $x^c$ to denote $\mk{a}(x)$. 
Identify $\R$ with the subalgebra of $A$ generated by $1$.  

For each element $x$ of $A$, the \emph{trace} of $x$ is $t(x):=x+x^c\in A$ and the (squared) \emph{norm} of $x$ is
$n(x):=xx^c\in A$. 

Let $d:=\dim_{\R}A$. Choose a real vector base $\mc{V}=(v_1,\ldots,v_d)$ of $A$ and define the norm $\| \cdot \|_{\mc{V}}:A \lra \R^+:=\{x\in\R\,|\, x\ge0\}$ by setting
\begin{equation} \label{eq:norm}
\textstyle
\|x\|_{\mc{V}}:=\left(\sum_{k=1}^dx_k^2\right)^{1/2},
\end{equation}
where $x_1,\ldots,x_d$ are the real coordinates of $x$ w.r.t.\ $\mc{V}$; that is, $x=\sum_{k=1}^dx_kv_k$. Evidently, the topology on $A$ induced by $\| \cdot \|_{\mc{V}}$ does not depend on the chosen base $\mc{V}$. We call such a topology on $A$ the \textit{euclidean topology on $A$}.

%\textit{In what follows, we assume that $A$ is equipped with the euclidean topology}. 

\emph{We assume that $A$ is equipped with a norm $\| \cdot \|_A$ satisfying the property: $\|x\|_A=\sqrt{n(x)}$ for each $x \in \Q_A$.} Since the real dimension of $A$ is finite, the topology induced by $\| \cdot \|_A$ on $A$ coincides with the euclidean one. 

We recall some definitions from  \cite{GhPe_Trends} and \cite{GhPe_AIM}.

\begin{definition}\label{cone}
%The \emph{normal cone} of $A$ is the subset
%\[\mc{N}_A:=\{0\}\cup\{x\in A\ |\ n(x)=n(x^c)\text{ is a real nonzero number}\}.\]
The \emph{quadratic cone} of  $A$ is the set
\[\Q_A:=\R\cup\{x\in A\ |\ t(x)\in\R,\ n(x)\in\R,\ 4n(x)> t(x)^2\}.\]
We also set\,
$\SS_A:=\{J\in \Q_A\ |\ J^2=-1\}$. Elements of\, $\SS_A$ are called \emph{square roots of $-1$} in the algebra $A$. For each $J\in \SS_A$, we will denote by $\C_J:=\langle 1,J\rangle\simeq\C$ the subalgebra of $A$ generated by  $J$. 
\end{definition}

The quadratic cone is a real cone invariant w.r.t.~translations along the real axis. Moreover, it has two fundamental properties  (cf.\ \cite[Propositions 1 and~3]{GhPe_AIM}):
\begin{itemize}
	\item
	$\Q_A$ coincides with the algebra $A$ if and only if  $A$ is isomorphic to one of the division algebras $\C,\q$ or $\oc$ with the usual conjugation mapping as anti-involution.
	\item
	$\Q_A=\bigcup_{J\in \SS_A}\C_J$ and $\C_I\cap\C_J=\R$  for every $I,J\in\SS_A$, $I\ne\pm J$. In particular, every $x\in\Q_A$ can be written in a unique way as $x=\re(x)+\im(x)$, where $\re(x):=(x+x^c)/2\in\R$ and $\im(x):=(x-x^c)/2$.
Moreover, every nonzero $x\in\Q_A$ has a multiplicative inverse $x^{-1}=n(x)^{-1}x^c\in\Q_A$. 
\end{itemize}

Fix a non--empty open subset $D$ of $\C$, invariant under the complex conjugation $z=\alpha+\ui \beta \longmapsto \overline{z}=\alpha-\ui \beta$. Let $\OO_D$ be the  subset of $\Q_A$ defined by:
\[
\OO_D:=\{x \in \Q_A \,|\, x=\alpha+\beta J, \ \alpha,\beta \in \R, \ \alpha+i\beta \in D, \ J \in \SS_A\}\:.
\]
A set of the form $\OO_D$ will be called \emph{circular}.
For simplicity, we assume $\OO_D$ connected. This is equivalent to require either that $D$ is connected if $D \cap \R \neq \emptyset$, or that $D$ consists of two connected components interchanged by the complex conjugation if $D \cap \R \neq \emptyset$. 

Let $\Ac=A\otimes_{\R}\C$ be the complexification of $A$. We will use the representation $\Ac=\{w=x+i y \ | \ x,y\in A\}$, with  $i^2=-1$ and complex conjugation $\overline w=\overline{x+iy}=x-iy$.
A function $F:D \lra \Ac$ is called a \textit{stem function} on $D$ if it satisfies the condition $F(\overline z)=\overline{F(z)}$ for each $z \in D$. If $F_1,F_2:D \lra A$ are the $A$-valued components of $F=F_1+iF_2$, then such a condition is equivalent to require that $F_1(\overline z)=F_1(z)$ and $F_2(\overline z)=-F_2(z)$ for each $z \in D$. We call $F$ \textit{continuous} if $F_1$ and $F_2$ are continuous. We say that $F$ is \textit{of class $\mscr{C}^1$} if $F_1$ and $F_2$ are of class $\mscr{C}^1$.

\begin{definition}
Any stem function $F=F_1+iF_2:D \lra \Ac$ induces a \emph{$($left$)$ slice function} $f=\I(F):\OO_D \lra A$ as follows: if $x=\alpha+\beta J\in \OO_D\cap\C_J$ for some $\alpha,\beta \in \R$ and $J \in \SS_A$, we set  
\[
f(x):=F_1(z)+JF_2(z) \quad (z=\alpha+i\beta).
\]
\end{definition}

We will denote by $\Sl^0(\OO_D,A)$ the real vector space of (left) slice functions on $\OO_D$ induced by continuous stem functions and  by $\Sl^1(\OO_D,A)$ %:=\{f=\I(F)\in\Sl^0(\OO_D,A)\ | \ F \in \mscr{C}^m(D,\Ac)\}$
the real vector space of slice functions induced by stem functions of class $\mscr{C}^1$.

Let  $f=\I(F) \in \mc{S}^1(\OO_D,A)$. Let us denote by $\partial F/\partial \overline{z}:D \lra \Ac$ the stem function on $D$ defined by
\[
\dd{F}{\overline{z}}:=\frac{1}{2}\left(\dd{F}{\alpha}+\ui \dd{F}{\beta} \right).
\]
It induces the \emph{slice derivatives} $\dd{f}{x}:=\I\left(\dd{F}{{z}}\right)$ and $\dd{f}{x^c}:=\I\left(\dd{F}{\overline{z}}\right)$ in $\mc{S}^0(\OO_D,A)$.

\begin{definition}
A slice function $f\in\mc{S}^1(\OO_D,A)$ is called \emph{slice regular} if it holds:
\[
\dd{f}{x^c}=0 \text{\quad on $\OO_D$} 
\]
i.e.\ if $f$ is induced by a holomorphic stem function $F:D\rightarrow A_\C$.
We denote by $\mc{SR}(\OO_D,A)$ the real vector space of all slice regular functions on~$\OO_D$.
\end{definition}

In general, the pointwise product of two slice functions is not a slice function. However, 
the pointwise product in the algebra $A\otimes\C$ induces a natural product on slice functions.
 
\begin{definition}
Let $f=\I(F),g=\I(G)$ be slice functions on $\OO_D$. The \emph{slice product} of $f$ and $g$ is the slice function on $\OO_D$
\[f\punto g:=\I(FG).\]
\end{definition}
If $f,g$ are slice regular, then also $f\punto g$ is slice regular. In general, $(f\punto g)(x)\ne f(x)g(x)$.
If the components $F_1,F_2$ of the \emph{first} stem function $F$ are real--valued,  or if $F$ and $G$ are both $A$--valued, then $(f\punto g)(x)= f(x)g(x)$ for every $x\in \OO_D$. In this case, we will use also the notation $fg$ in place of $f\punto g$.

Given $y \in\Q_A$ and $n \in \N$, we denote by $(x-y)^{\punto n}$ the $n$-th power of the slice function $x-y$ w.r.t.\ the slice product. As an immediate consequence of the definitions, polynomials  of the form $\sum_n^d(x-y)^{\punto n}a_n$, with coefficients $a_n$ in $A$, define slice regular functions on $\Q_A$.
More generally, we are interested in power series $\sum_n(x-y)^{\punto n}a_n$ with right coefficients in $A$. When $y$ is real, then the slice power $(x-y)^{\punto n}$ coincides with the usual power $(x-y)^n$ and the series converges on the intersection of an euclidean ball $B(y,R)=\{x\in A\;|\;\|x\|_A<R\}$ with the quadratic cone. But when $y\in\Q_A\setminus\R$, the convergence of the series reveals unexpected phenomena. As already seen in the quaternionic case \cite{GeSto2012MathAnn}, the sets of convergence of series $\sum_n(x-y)^{\punto n}a_n$ may have empty interior w.r.t.\ the euclidean topology of $\Q_A$. 
To express more precisely this aspect, we introduce a metric on $\Q_A$, using the same approach of \cite{GeSto2012MathAnn}. 

Let $x,y \in \Q_A$ with $y=\xi+J\eta$ for some $\xi,\eta \in \R$ and $J \in \cS_A$. Define
\begin{equation*}
\sigma_A(x,y):=
\begin{cases}
 \|x-y\|_A &\text{\quad if }x\in \C_J\\
\sqrt{{|\re(x)-\re(y)|}^2+{\left(\|\im(x)\|_A+\|\im(y)\|_A\right)}^2} &\text{\quad if }x\notin \C_J.
\end{cases}
\end{equation*}
The topology of $\Q_A$ induced by $\sigma_A$ is finer than the euclidean one. If $\cS_A$ has no isolated points, a $\sigma_A$--ball  of radius $r$ centered at $y$ has empty interior if $r\le|\im(y)|$. In Sect.~\ref{Abel_Theorem_for_power_and_spherical_series} we show that the sets of convergence of series $\sum_n(x-y)^{\punto n}a_n$ are $\sigma_A$--balls centered at $y$, and obtain a version of the Abel Theorem for these series. We also give formulas for the coefficients of the expansion.
We then show in Sect.~\ref{Power_expansion_for_slice_regular_functions} that when a function is defined on an open subset of the quadratic cone, then its analyticity w.r.t.\ the metric $\sigma_A$ is equivalent to slice regularity.

\begin{definition}
Given a function $f:U \lra A$ defined on a non--empty open subset $U$ of $\Q_A$, we say that $f$ is \emph{$\sigma_A$--analytic} or \emph{power analytic},  if, for each $y \in U$, there exists a non--empty $\sigma_A$--ball $\Sigma$ centered at $y$ and contained in $U$, and a series $\sum_{n \in \N}(x-y)^{\punto n}a_n$ with coefficients in $A$, which converges to $f(x)$ for each $x \in \Sigma \cap U$. 
\end{definition}

The main result of Sect.~\ref{Power_expansion_for_slice_regular_functions} is the following.

\begin{theorem}
Let $\OO_D$ be connected and let $f:\OO_D \lra A$ be any function. The following assertions hold.
\begin{itemize}
 \item[$(\mr{i})$] If $D \cap \R=\emptyset$, then $f$ is a slice regular function if and only if $f$ is a $\sigma_A$--analytic slice function.
 \item[$(\mr{ii})$] If $D \cap \R \neq \emptyset$, then $f$ is a slice regular function if and only if $f$ is $\sigma_A$--analytic.
\end{itemize}
\end{theorem}

Sect.~\ref{Power_expansion_for_slice_regular_functions} contains also estimates for the coefficients of the power expansion of a slice regular function and the expression of the remainder in integral form. These results are new also in the quaternionic case.

Instead of using slice powers $(x-y)^{\punto n}$, other classes of functions can be considered for series expansions on the quadratic cone. A natural choice is given by the powers of the characteristic polynomial $\Delta_y(x):=(x-y)\cdot(x-y^c)$ of an element $y$ of $\Q_A$. More precisely, for each $m\in\N$ we define, following \cite{StoppatoAdvMath2012}, the slice regular polynomial functions 
\[
\stx_{y,2m}(x):=\Delta_y(x)^m,
\quad
\stx_{y,2m+1}(x):=\Delta_y(x)^m(x-y).
\]
Differently from (slice) power series,  series of type  $\sum_{n \in \N}\stx_{y,n}(x)s_n$  have convergence sets that are always open w.r.t.\ the euclidean topology. More precisely, these sets are related to a pseudo--metric defined on the quadratic cone, called \emph{Cassini pseudo--metric}.
If $x$ and $y$ are points of $\Q_A$, then we set, in analogy with \cite{StoppatoAdvMath2012},
\[
\sto_A(x,y):=\sqrt{\|\Delta_y(x)\|_A}.
\]
The function $\sto_A$ turns out to be a pseudo--metric on $\Q_A$, whose induced topology is strictly coarser than the euclidean one. 
In Sect.~\ref{Abel_Theorem_for_power_and_spherical_series} we show that the sets of convergence of series $\sum_{n \in \N}\stx_{y,n}(x)s_n$ are $\sto_A$--balls centered at $y$. We also give the corresponding Abel Theorem and formulas for computing the coefficients of the expansion, which are new also in the quaternionic case.

\begin{definition}
Given a function $f:V \lra A$ from a non--empty circular open subset $V$ of $\Q_A$ into $A$, we say that $f$ is \emph{$\sto_A$--analytic} or \emph{spherically analytic}, if, for each $y \in V$, there exists a non--empty $\sto_A$--ball $\Sto$ centered at $y$ and contained in $V$, and a series $\sum_{n \in \N}\stx_{y,n}(x)s_n$ with coefficients in $A$, which converges to $f(x)$ for each $x \in \Sto \cap V$.
\end{definition}
In the quaternionic case \cite{StoppatoAdvMath2012}, spherically analytic functions were called \emph{symmetrically analytic}.
We finally show in Sect.~\ref{Spherical_expansion_for_slice_regular_functions} that a function is spherically analytic on a circular domain if, and only if, it is slice regular therein.

\begin{theorem}\label{spherical_intro}
Let $\OO_D$ be connected and let $f:\OO_D \lra A$ be any function. The following assertions hold.
\begin{itemize}
 \item[$(\mr{i})$] If $D \cap \R=\emptyset$, then $f$ is a slice regular function if and only if $f$ is slice and spherically analytic.
 \item[$(\mr{ii})$] If $D \cap \R \neq \emptyset$, then $f$ is a slice regular function if and only if $f$ is spherically analytic.
\end{itemize}
\end{theorem}

Sect.~\ref{Spherical_expansion_for_slice_regular_functions} contains also estimates for the coefficients of the spherical expansion of a slice regular function and the integral expression of the remainder, with the related estimate in terms of the pseudo--metric $\sto_A$.
The estimates, the integral expression of the remainder and point (i) of Theorem~\ref{spherical_intro}, are new also in the quaternionic case.

\section{Preliminary results}
\label{Preliminary_results}

\subsection*{The Cullen derivative}

In Definition 2.1 of \cite{GeSt2006CR}, Gentili and Struppa introduced the notion of Cullen derivative $\partial_Cf$ of a slice regular function $f$ on a domain of the quaternionic space $\q$ (see also Definition 2.2 of \cite{GeSt2007Adv}). The same definition can be given on a subdomain $\OO_D$ of the quadratic cone of an alternative $^*$-algebra. Moreover, on $\OO_D\setminus\R$, the  definition can be extended to any real differentiable function.
%As we will see in Section~\ref{Preliminary_results},  the slice derivative $\partial f/\partial x$ coincides with $\partial_Cf$ on $\OO_D \setminus \R$.

Given $J \in \cS_A$ and $f \in \mc{S}^0(\OO_D,A)$, we denote by $\Phi_J:\C \lra \Q_A$ and $f_J:D \lra A$ the functions defined by setting
\begin{equation} \label{eq:defs}
\Phi_J(\alpha+\ui\beta):=\alpha+J\beta
\; \; \mbox{ and } \; \;
f_J(z):=f(\Phi_J(z))\:.
\end{equation}

%We have:

\begin{lemma} \label{lem:cullen}
Let $f \in \mc{S}^1(\OO_D,A)$ and let $J \in \cS_A$. Then, for each $w \in D$, it holds:
\begin{equation} \label{eq:cullen^1}
\dd{f}{x}(\Phi_J(w))=\dd{f_J}{z}(w)
\; \mbox{ and } \; \; \dd{f}{x^c}(\Phi_J(w))=\dd{f_J}{\overline{z}}(w)\:,
\end{equation}
where $\partial/\partial z:=(1/2)(\partial/\partial \alpha-J \cdot \partial/\partial \beta)$ and $\partial/\partial \overline{z}:=(1/2)(\partial/\partial \alpha+J \cdot \partial/\partial \beta)$.
Furthermore, if $f \in \mc{S}^{\infty}(\OO_D,A)$ and $n \in \N$, then
\begin{equation} \label{eq:cullen^n}
\dd{^nf}{x^n}(\Phi_J(w))=\dd{^nf_J}{z^n}(w).
\end{equation}
\end{lemma}
\begin{proof}
Let $F=F_1+\ui F_2:D \lra A \otimes \C$ be the stem function of class $\mscr{C}^1$ inducing $f$, let $w \in D$ and let $y:=\Phi_J(w)$.
Let $\partial_{\alpha}$ and $\partial_{\beta}$ denote $\partial/\partial \alpha$ and $\partial/\partial \beta$, respectively. 
 Since $2 \, \partial F/\partial z=(\partial_\alpha F_1+\partial_\beta F_2)+\ui(\partial_\alpha F_2-\partial_\beta F_1)$ and $f_J=F_1+JF_2$, we have:
\begin{align*}
2 \, \dd{f}{x}(y)=& \, 2 \, \I\left(\dd{F}{z}\right)(y)=
%\left(\dd{F_1}{\alpha}(w)+\dd{F_2}{\beta}(w)\right)+J\left(\dd{F_2}{\alpha}(w)-\dd{F_1}{\beta}(w)\right)=\\
%=&
%\left(\dd{F_1}{\alpha}(w)+
%J\dd{F_2}{\alpha}(w)\right)-J\left(\dd{F_1}{\beta}(w)+J\dd{F_2}{\beta}(w)\right)=\\
%=& \left(\dd{f_J}{\alpha}-J\dd{f_J}{\beta}\right)(w)=2 \, \dd{f_J}{z}(w).
\left(\partial_\alpha{F_1}(w)+\partial_\beta{F_2}(w)\right)+J\left(\partial_\alpha{F_2}(w)-\partial_\beta{F_1}(w)\right)=\\
=&
\left(\partial_\alpha F_1 (w)+
J\partial_\alpha{F_2}(w)\right)-J\left(\partial_\beta{F_1}(w)+J\partial_\beta{F_2}(w)\right)=\\
=& \left(\partial_\alpha{f_J}-J\partial_\beta{f_J}\right)(w)=2 \, \dd{f_J}{z}(w).
\end{align*}
Similarly, one can prove that $(\partial f/\partial x^c)(y)=(\partial f_J/\partial\overline{z})(w)$ and for every $f \in \mc{S}^{\infty}(\OO_D,A)$, $(\partial^n f/\partial x^n)(y)=(\partial^n f_J/\partial z^n)(w)$.
\end{proof}

At a point $y=\Phi_J(w)\in\OO_D \setminus \R$, the Cullen derivative of a slice function $f$ is defined as $(\partial f_J/\partial{z})(w)$. Therefore the slice derivative $\partial f/\partial x$ coincides with $\partial_Cf$ on $\OO_D \setminus \R$.

\subsection*{The splitting property.} 
Suppose $\cS_A \neq \emptyset$. Since the elements of $\cS_A$ are square roots of~$-1$, the left multiplications by such elements are complex structures on $A$. In particular, it follows that %, if $\cS_A \neq \emptyset$, then 
the real dimension of $A$ is even and positive.
%\textit{In what follows, we assume that $\cS_A \neq \emptyset$}.
Let $h$ be the non-negative integer such that $d=\dim_{\R}A=2h+2$.

\begin{definition}
We say that $A$ has the \emph{splitting property} if, for each $J \in \cS_A$, there exist $J_1,\ldots,J_h \in A$ such that $(1,J,J_1,JJ_1,\ldots,J_h,JJ_h)$ is a real vector base of $A$, called a \emph{splitting base of $A$ associated with $J$}.
\end{definition}

The following result ensures that every real alternative algebra we are considering has this property.

\begin{lemma} \label{lem:fundamental}
Every real alternative algebra of finite dimension with unity, equipped with an anti-involution such that $\cS_A \neq \emptyset$, has the splitting property.
\end{lemma}
\begin{proof}
Let $J\in\cS_A$. Denote by $A_J$ the complex vector space defined on $A$ by the left multiplication by $J$ and let $h+1=\dim_{\C}A_J$. Any complex basis $(1,J_1,\ldots,J_h)$ of $A_J$ defines a real basis $(1,J,J_1,JJ_1,\ldots,J_h,JJ_h)$ of $A$. 
\end{proof}

We have the following general splitting lemma.

\begin{lemma} \label{lem:splitting}
Let $f \in \mc{SR}(\OO_D,A)$, let $J \in \cS_A$, let $(1,J,J_1,JJ_1,\ldots,J_h,JJ_h)$ be a splitting base of $A$ associated with $J$ and let $f_{1,0},f_{2,0},\ldots,f_{1,h},f_{2,h}$ be the functions in $\mscr{C}^1(D,\R)$ such that $f_J=\sum_{\ell=0}^h(f_{1,\ell}J_{\ell}+f_{2,\ell}JJ_{\ell})$%=\sum_{\ell=0}^h(f_{1,\ell}+f_{2,\ell}J)J_{\ell}$
, where $J_0:=1$. Then, for each $\ell \in \{0,1,\ldots,h\}$, $f_{1,\ell}$ and $f_{2,\ell}$ satisfy the following Cauchy--Riemann equations:
\[
\dd{f_{1,\ell}}{\alpha}=\dd{f_{2,\ell}}{\beta}
\; \; \mbox{ and }\; \;
\dd{f_{1,\ell}}{\beta}=-\dd{f_{2,\ell}}{\alpha}.
\]
\end{lemma}
\begin{proof}
Let $w \in D$ and let $y:=\Phi_J(w)$. By Lemma \ref{lem:cullen}, we have that $(\partial f/\partial x^c)(y)=(\partial f/\partial \overline{z})(w)$. Since $(\partial f/\partial x^c)(y)=0$, we infer that 
\begin{align*}
0=& \, 2 \, (\partial f/\partial \overline{z})(w)=
\textstyle \big(\partial_{\alpha}+J\partial_{\beta}\big)f_J(w)=\sum_{\ell=0}^h\big((\partial_{\alpha}f_{1,\ell}(w))J_{\ell}+(\partial_{\alpha}f_{2,\ell}(w))JJ_{\ell}\big)+\\
&
\textstyle
+J\sum_{\ell=0}^h\big((\partial_{\beta}f_{1,\ell}(w))J_{\ell}+(\partial_{\beta}f_{2,\ell}(w))JJ_{\ell}\big)=\\
=& \,
\textstyle
\sum_{\ell=0}^h\big((\partial_{\alpha}f_{1,\ell}(w)-\partial_{\beta}f_{2,\ell}(w))J_{\ell}+(\partial_{\beta}f_{1,\ell}(w)+\partial_{\alpha}f_{2,\ell}(w))JJ_{\ell}\big),
\end{align*}
where $\partial_{\alpha}$ and $\partial_{\beta}$ denote $\partial/\partial \alpha$ and $\partial/\partial \beta$, respectively. It follows that $\partial_{\alpha}f_{1,\ell}(w)=\partial_{\beta}f_{2,\ell}(w)$ and $\partial_{\beta}f_{1,\ell}(w)=-\partial_{\alpha}f_{2,\ell}(w)$, as desired.
\end{proof}

\begin{remark}\label{sequence}
An immediate consequence of the previous Lemma is that, given a sequence $\{f_n\}_{n\in\N}$ of slice regular functions on $\OO_D$, uniformly convergent on compact subsets of $\OO_D$,  the limit of the sequence is slice regular on $\OO_D$.
\end{remark}

\subsection*{The norm $\boldsymbol{\| \cdot \|_A}$.} 
We recall that we are assuming that $A$ is equipped with a norm $\| \cdot \|_A$ having the following property:
\begin{equation} \label{eq:E}
\|x\|_A=\sqrt{n(x)} \; \; \mbox{ for each }x \in \Q_A\:.
\end{equation}
Note that for each $J\in\cS_A$ and each $z\in\C$, it holds $\|\Phi_J(z)\|=|z|$, since $\Phi_J(z)\in\Q_A$.

\begin{lemma} \label{lem:universal-norm}
Let $A$ be a real alternative algebra of finite dimension with unity, equipped with an anti--involution $\mk{a}$ and with the euclidean topology. Suppose that $\cS_A \neq \emptyset$ and there exists a norm $\| \cdot \|_A$ on $A$ satisfying $(\ref{eq:E})$. Then there exists a positive real constant $\mr{H}=\mr{H}(\mk{a},\| \cdot \|_A)$, depending only on $\mk{a}$ and on $\| \cdot \|_A$, with the following property: for~each $J \in \cS_A$, it is possible to find a splitting base $\mscr{B}_J=(1,J,J_1,JJ_1,\ldots,J_h,JJ_h)$ of $A$ associated with~$J$ such that $\|J_{\ell}\|_A=1$ for each $\ell \in \{1\ldots,h\}$ and $\|x\|_{\mscr{B}_J} \leq \mr{H} \, \|x\|_A$ for each $x \in A$.
%with the following property: for each $\vep \in \R^+$ and for each $J \in \cS_A$, there exists a splitting base $\mscr{B}=(1,J,J_1,JJ_1,\ldots,J_h,JJ_h)$ of $A$ associated with $J$ such that $\|J_{\ell}\|_A=1$ for each $\ell \in \{1\ldots,h\}$ and
%\begin{equation} \label{eq:H}
%\|x\|_{\mscr{B}} \leq (\mr{H}+\vep)\|x\|_A \; \mbox{ for each }x \in A.
%\|x\|_{\mscr{B}_J} \leq \mr{H} \, \|x\|_A \; \mbox{ for each }x \in A.
%\end{equation}
\end{lemma}
\begin{proof}
Let $\mk{B}$ be the set of all real vector bases of $A$ and let $\mc{V}=(v_1,\ldots,v_d) \in \mk{B}$. 
%For each $x=\sum_{k=1}^dx_kv_k \in A$ with $x_1,\ldots,x_d \in \R$, we have that
%\[
%\textstyle
%\|x\|_A \leq \sum_{k=1}^d|x_k| \, \|v_k\|_A \leq \|x\|_{\mc{V}}\left(\sum_{k=1}^d\|v_k\|_A^2\right)^{1/2}. 
%\]
%It follows that $\| \cdot \|_A:A \lra \R^+$ is continuous. Since $\mc{S}_{\mc{V}}:=\{x \in A \, | \, \|x\|_{\mc{V}}=1\}$ is compact and $\| \cdot \|_A$ is positive on $\mc{S}_{\mc{V}}$, we have that $\| \cdot \|_A$ assumes positive minimum on $\mc{S}_{\mc{V}}$. 
Since the norm $\| \cdot \|_A$ is continuous and the set $\mc{S}_{\mc{V}}:=\{x \in A \, | \, \|x\|_{\mc{V}}=1\}$ is compact, we have that $\| \cdot \|_A$ assumes positive minimum on $\mc{S}_{\mc{V}}$. 
Define:
\[
\textstyle
m(\mc{V},\| \cdot \|_A):=\big(\min_{\, x \in  \mc{S}_{\mc{V}}}\|x\|_A\,\big)^{-1}>0\:.
%\; \; \; \mbox{ and }\; \; \; M(\mc{V},\| \cdot \|_A):=\max_{x \in  \mc{S}_{\mc{V}}}\|x\|_A>0.
\]
By homogeneity, we infer at once that
\begin{equation} \label{eq:estimate-1}
\|x\|_{\mc{V}} \leq m(\mc{V},\| \cdot \|_A) \, \|x\|_A %\leq M(\mc{V},\| \cdot \|_A) \, \|x\|_{\mc{V}}.
\; \mbox{ for each }x \in A\:.
\end{equation}
Since $\mk{a}$ is 
%a real linear map, $\mk{a}$ is also 
continuous, % It follows that
 the trace $t$ and the squared norm $n$ are continuous as well. By Proposition 1 of \cite{GhPe_AIM}, we have that $\cS_A=t^{-1}(0) \cap n^{-1}(1)$. On the other hand, by (\ref{eq:E}), $\cS_A$ is contained in the compact subset $\mc{S}_A:=\{x \in A \, | \, \|x\|_A=1\}$ of $A$. In this way, we infer that $\cS_A$ is compact as well, because it is closed in $A$ and contained in $\mc{S}_A$.

Denote by $\mk{S}$ the set of all continuous maps $\mc{J}:\cS_A \lra \mc{S}_A^h$ such that, for each $J \in \cS_A$, the $d$--uple $\mscr{B}({\mc{J},J}):=(1,J,J_1(J),J \cdot J_1(J),\ldots,J_h(J),J \cdot J_h(J))$ is a splitting base of $A$ associated with $J$, where $(J_1(J),\ldots,J_h(J))=\mc{J}(J)$. We show that $\mk{S} \neq \emptyset$.
For every fixed $J\in\cS_A$, let $(1,J,J_1(J),J \cdot J_1(J),\ldots,J_h(J),J \cdot J_h(J)$ be a splitting base of $A$ associated with $J$. Denote by $A_J$ the complex vector space defined on $A$ by the left multiplication by $J$. By continuity, for every $J' \in \cS_A$ sufficiently close to $J$, the set  $(1,J',J_1,J'J_1,\ldots,J_h,J'J_h)$ is still a $\R$--basis of $A$, and therefore $(1,J_1,\ldots,J_h)$ is a $\C$--basis of $A_{J'}$. This means that, locally near $J$, we can define the map $\mc{J}:\cS_A \lra  \mc{S}_A^h$ as the constant mapping $J'\mapsto (J_1,\ldots,J_h)$ after normalization w.r.t.~$\|\cdot\|_A$. A partition of unity argument allows to conclude that $\mk{S} \neq \emptyset$.

Let $\mc{J}=(J_1,\ldots,J_h) \in \mk{S}$ and let $\mk{M}(\mc{V},\mc{J}):\cS_A \lra \R^{d \times d}$ be the continuous map, sending $J$ into the real $d \times d$ matrix $\mk{M}(\mc{V},\mc{J})(J)$ having as coefficients %$\mk{m}_{ij}(\mc{V},\mc{J})$
 of the $j^{\mr{th}}$-column the coordinates of $v_j$ w.r.t.~$\mc{J}(J)$. %Denote by $\mk{m}_{ij}(\mc{V},\mc{J}):\cS_A \lra \R$ the $(i,j)$--coefficient of $\mk{M}(\mc{V},\mc{J})$, which is a continuous function. 
Let $\|\cdot\|_2$ denote the $L^2$ matrix norm and define:
\[
%M(\mc{V},\mc{J},\mk{a}):=\sqrt{d} \cdot \max_{\cS_A}\max_{i,j}\left|\mk{m}_{ij}(\mc{V},\mc{J})\right|>0\:.
M(\mc{V},\mc{J},\mk{a}):=\max_{\cS_A} \|\mk{M}(\mc{V},\mc{J})(J)\|_2>0\:.
\]
Let $J \in \cS_A$ and let $x \in A$. %Write $x$ as follows: $x=\sum_{k=1}^dx_kv_k=y_1+y_2J+y_3J_1(J)+y_4J \cdot J_1(J)+\dots+y_{d-1}J_h(J)+y_dJ \cdot J_h(J)$ with $x_1,\ldots,x_d,y_1,\ldots,y_d \in \R$. 
Denote by $\mr{x}$ and $\mr{y}$ the  column vectors %$(x_1,\ldots,x_d)$ and $(y_1,\ldots,y_d)$, 
of the coordinates of $x$ w.r.t.~the bases $\mc{V}$ and $\mscr{B}_{\mc{J}(J)}$, respectively. Since $\mr{y}=\mk{M}(\mc{V},\mc{J})(J) \cdot \mr{x}$, we infer that
\begin{equation} \label{eq:estimate-2}
\textstyle
\|x\|_{\mscr{B}_{\mc{J}(J)}}%=\left(\sum_{k=1}^dy_k^2\right)^{1/2}
\leq %M(\mc{V},\mc{J},\mk{a})\left(\sum_{k=1}^dx_k^2\right)^{1/2}=
M(\mc{V},\mc{J},\mk{a})\|x\|_{\mc{V}}\:.
\end{equation}
By combining (\ref{eq:estimate-1}) with (\ref{eq:estimate-2}), we obtain that
\[
\|x\|_{\mscr{B}_{\mc{J}(J)}} \leq M(\mc{V},\mc{J},\mk{a}) \cdot m(\mc{V},\| \cdot \|_A) \cdot \|x\|_A
\]
for each $J \in \cS_A$ and for each $x \in A$. 
In particular, taking $x=J$ we get that $ M(\mc{V},\mc{J},\mk{a}) \cdot m(\mc{V},\| \cdot \|_A)\ge1$.
Now it suffices to define $\mr{H}(\mk{a},\| \cdot \|_A)$ as follows:
\[
\mr{H}(\mk{a},\| \cdot \|_A):=2\left(\inf_{\mc{V} \in \mk{B},\mc{J} \in \mk{S}}M(\mc{V},\mc{J},\mk{a}) \cdot m(\mc{V},\| \cdot \|_A)\right)>0.
\]
The proof is complete.
\end{proof}

\begin{remark}%(approfondimento non essenziale)
If $A=\q$ or $\oc$, with $x^c=\bar x$ the usual conjugation and $\| x \|_A=\sqrt{x\bar x}$ the euclidean norm of $\R^4$ and $\R^8$, respectively, the constant $\mr{H}$ in the previous Lemma can be taken equal to 1. Given $J\in\cS_A$, there exists an orthonormal splitting basis $\mscr{B}_J$ obtained by completing the set $\{1,J\}$ to a orthonormal basis.
%\marginpar{citazioni}
 Therefore $\|x\|_{\mscr{B}_J}=\|x\|_A$ for every $x\in A$. 

If $A$ is the real Clifford algebra $\R_n$ with signature $(0,n)$, with $x^c$ the Clifford conjugation, there are (at least) two norms satisfying condition \eqref{eq:E}: the euclidean norm induced by $\R^{2^n}$ and the \emph{Clifford operator norm}, as defined in \cite[(7.20)]{GilbertMurray}.

%\footnote{Questo $\mr{H}$ sarebbe quello ``ottimo'' con $\vep>0$}
\end{remark}

Given an element $a$ of $A$ and a sequence $\{a_n\}_{n \in \N}$ in $A$, we will write $a=\sum_{n \in \N}a_n$ meaning that the series $\sum_{n \in \N}a_n$ converges to $a$ in $A$. This is equivalent to say that $\lim_{n \rightarrow +\infty}\|a-\sum_{k=0}^na_k\|_A=0$.
The reader observes that, being the real dimension of $A$ finite, the topology induced by $\| \cdot \|_A$ on $A$ coincides with the euclidean one.

\subsection*{The metric $\boldsymbol{\sigma_A}$ on $\boldsymbol{\Q_A}$.} The definition of the metric $\sigma$ on $\q$ given in \cite{GeSto2012MathAnn} can be mimicked to define the following function $\sigma_A:\Q_A \times \Q_A \lra \R^+$. Let $x,y \in \Q_A$ with $y=\xi+J\eta$ for some $\xi,\eta \in \R$ and $J \in \cS_A$. Define
\begin{equation*}
\sigma_A(x,y):=
\begin{cases}
 \|x-y\|_A &\text{\quad if }x\in \C_J\\
\sqrt{{|\re(x)-\re(y)|}^2+{\left(\|\im(x)\|_A+\|\im(y)\|_A\right)}^2} &\text{\quad if }x\notin \C_J
\end{cases}
\end{equation*}
Evidently, $\sigma_{\q}$ coincides with $\sigma$. Let $\alpha,\beta \in \R$ and $I \in \cS_A$ such that $x=\alpha+I\beta$. Define $z:=\alpha+\ui\beta$ and $w:=\xi+\ui\eta$. Proceeding as in the proof of Lemma 1 of~\cite{GeSto2012MathAnn}, we obtain:
\begin{equation} \label{eq:max}
\sigma_A(x,y)=
\begin{cases}
 |z-w| &\text{\quad if }x\in \C_J\\
\max\{|z-w|,|z-\overline{w}|\} &\text{\quad if }x\notin \C_J
\end{cases}
\end{equation}
%where $|v|$ denotes the usual norm of the complex number $v$.
Furthermore, by using (\ref{eq:max}) exactly as in Section 3 of \cite{GeSto2012MathAnn}, we see that \textit{$\sigma_A$ is a metric on $\Q_A$}. 

Let $r \in \R^+$ and let $\Sigma_A(y,r)$ be the $\sigma_A$--ball of $\Q_A$ centered at $y$ of radius $r$; that is, $\Sigma_A(y,r):=\{p \in \Q_A \, | \, \sigma_A(p,y)<r\}$. Observe that the subset of $\C$ consisting of points $z$ such that $\max\{|z-w|,|z-\overline{w}|\}<r$ is equal to the intersection $B(w,r) \cap B(\overline{w},r)$, where $B(w,r)$ is the open ball of $\C$ centered at $w$ of radius~$r$. In this way, if we define, for $J\in\cS_A$ such that $y\in\C_J$,
\begin{equation} \label{eq:J-ball}
B_J(y,r):=\{p \in \C_J \,  | \, \|p-y\|_A<r\}%n(p-y)<r^2\}
\end{equation}
and
\begin{equation} \label{eq:OO-ball}
\OO(y,r) \mbox{ as the circularization of }B(w,r) \cap B(\overline{w},r),
\end{equation}
then we obtain that
\begin{equation} \label{eq:sigma-ball}
\Sigma_A(y,r)=B_J(y,r) \cup \OO(y,r).
\end{equation}
On the reduced quaternions $\R+i\R+j\R$, such a decomposition of $\Sigma_A(y,r)$ can be concretely visualized: see Figure 2 of \cite{GeSto2012MathAnn}.

Let us prove that $\|x-y\|_A \leq \sigma_A(x,y)$. If $x \in \C_J$, then $\|x-y\|_A=\sigma_A(x,y)$. Suppose $x \not\in \C_J$ and hence $x,y \not\in \R$. Up to replace $I$ with $-I$ and $J$ with $-J$, we may assume that $\beta$ and $\eta$ are positive. In this way, by (\ref{eq:max}), we have that $\sigma_A(x,y)=|z-\overline{w}|$. Denote by $p$ the real number aligned with $z$ and $\overline{w}$. By combining the equality $\sigma_A(x,y)=|z-p|+|p-\overline{w}|$ and (\ref{eq:E}), we infer that
\[
\|x-y\|_A \leq \|x-p\|_A+\|p-y\|_A=|z-p|+|p-w|=|z-p|+|p-\overline{w}|=\sigma_A(x,y),
\]
as desired. It follows that the \textit{topology of $\Q_A$ induced by $\sigma_A$ is finer than the euclidean one}. Thanks to (\ref{eq:sigma-ball}), they coincide if and only if $\cS_A$ is discrete w.r.t.~the euclidean topology.

\subsection*{The Cassini pseudo-metric $\boldsymbol{\sto_A}$ on $\boldsymbol{\Q_A}$.} 
For each $w \in \C$, denote by $\Delta_w:\C \lra \C$ the polynomial function defined by setting
\[
\Delta_w(z):=(z-w)(z-\overline{w})=z^2-z(w+\overline{w})+w\overline{w}.
\]
%Let $\C^+:=\{\alpha+\ui \beta \in \C \, | \, \beta \geq 0\}$ and 
Let $\sto:\C \times \C \lra \R^+$ be the function defined as follows:%\footnote{Pu\`o darsi che, nel proseguo, ci sia un p\`o di confusione con la $\sto$: in un primo momento era stata definita solo su $\C^+$.}
\[
\sto(z,w):=%\sqrt{|z-w||z-\overline{w}|}=\sqrt{|z^2-z(w+\overline{w})+w\overline{w}|}.
\sqrt{|\Delta_w(z)|}.
\]
Let $r \in \R^+$ and let $\Sto(w,r):=\{z \in \C\, | \, \sto(z,w)<r\}$. %Figure~1 represents $\Sto(w,r)$ in the following four situations:  $r<\im(w)$, $r=\im(w)$, $r>\im(w)>0$, $w \in \R$. 
The boundaries of the $\sto$--balls $\Sto(w,r)$ are \emph{Cassini ovals} for $w\notin\R$ and circles for $w\in\R$. 
%$0<\im(w)<r$, $\im(w)=r$ and $\im(w)>r$.

%\begin{center}
%\textit{Figure 1 (forse?)}
%\end{center}

Define $\C^+:=\{\alpha+\ui \beta \in \C \, | \, \beta \geq 0\}$.
%The function $\sto$ is a metric on $\C^+$. Let us prove this assertion .......
\begin{lemma} \label{lem:tau}
The function $\sto$ is a pseudo--metric on $\C$ and its restriction on $\C^+ \times \C^+$ is a metric on $\C^+$.
\end{lemma}
\begin{proof}
Firstly, observe that $\sto(z,w)=\sto(z,\overline w)$ and $\sto(z,w)=0$ if and only if $z\in\{w,\overline w\}$.
Secondly, the function $\sto$ is symmetric: $|\Delta_w(z)|=|z-w||z-\overline w|=|w-z||w-\overline z|=|\Delta_z(w)|$.
In order to prove the triangle inequality, assume that $z,w,y\in\C^+$ and let $\alpha:=|z-w|$, $\beta=|z-y|$, $\gamma:=|w-y|$. Then $\alpha\le\alpha':=|z-\overline w|$, $\beta\le\beta':=|z-\overline y|$, $\gamma\le\gamma':=|y-\overline w|$.
Assume that $\beta+\gamma'\le\beta'+\gamma$ (otherwise swap $z$ and $w$). Since $\alpha=|z-w|\le|z-y|+|y-w|=\beta+\gamma$ and $\alpha'=|z-\overline w|\le|z-y|+|y-\overline w|=\beta+\gamma'$, we get
\begin{align*}
\alpha\alpha'\le(\beta+\gamma)(\beta+\gamma')&=\beta(\beta+\gamma')+\beta\gamma+\gamma\gamma'\le \beta(\beta'+\gamma)+\beta\gamma+\gamma\gamma'\\
&\le \beta\beta'+2\sqrt{\beta\beta'\gamma\gamma'}+\gamma\gamma'=(\sqrt{\beta\beta'}+\sqrt{\gamma\gamma'})^2\:.
\end{align*}
Therefore $\sqrt{|\Delta_w(z)|}=\sqrt{\alpha\alpha'}\le \sqrt{\beta\beta'}+\sqrt{\gamma\gamma'}=
\sqrt{|\Delta_y(z)|}+\sqrt{|\Delta_w(y)|}$.
\end{proof}

It is important to observe that, for every $z,w\in\C$, it holds:
\begin{equation} \label{eq:sto-inequality}
\sqrt{|\Delta_w(z)|+|\im(w)|^2}-|\im(w)| \leq |z-w| \leq \sqrt{|\Delta_w(z)|+|\im(w)|^2}+|\im(w)|.
\end{equation}
In fact, if $|z-w|>\sqrt{|\Delta_w(z)|+|\im(w)|^2}+|\im(w)|$, then
\[
|z-\overline{w}| \geq |z-w|-|2\im(w)|>\sqrt{|\Delta_w(z)|+|\im(w)|^2}-|\im(w)|
\]
and hence $|\Delta_w(z)|=|z-w||z-\overline{w}|>|\Delta_w(z)|$, which is impossible. The inequality $\sqrt{|\Delta_w(z)|+|\im(w)|^2}-|\im(w)| \leq |z-w|$ can be proved in a similar way. A generalized version of this argument was employed in \cite{StoppatoAdvMath2012}  (see Lemma~2.3) to prove the quaternionic version of the preceding inequalities.

\begin{definition}
Let $\sto_A:\Q_A \times \Q_A \lra \R^+$ be the function  defined as follows. If $x=\alpha+I\beta$ and $y=\xi+J\eta$ are points of $\Q_A$ with $\alpha,\beta,\xi,\eta \in \R$  and $I,J \in \cS_A$, then we set
\[
\sto_A(x,y):=\sto(z,w),
\]
where $z:=\alpha+\ui\beta$ and $w:=\xi+\ui\eta$. 
\end{definition}

Thanks to Lemma~\ref{lem:tau}, $\sto_A$ is symmetric w.r.t.\ $x$ and~$y$, and it satisfies the triangle inequality. Furthermore, $\sto_A(x,y)=0$ if and only if $\cS_x=\cS_y$. In conclusion, it turns out that \textit{$\sto_A$ is a pseudo--metric on $\Q_A$}, which will be called the \emph{Cassini pseudo--metric} on $\Q_A$.

Denote by $\Sto_A(y,r)$ the $\sto_A$--ball of $\Q_A$ centered at $y$ of radius $r$; that is, $\Sto_A(y,r):=\{x \in \Q_A \, | \, \sto_A(x,y)<r\}$. By definition of $\sto_A$, it follows immediately that
%\begin{equation} \label{eq:sto}
$\Sto_A(y,r)$ is the circularization of $\Sto(w,r)$.
%\end{equation}
%On the reduced quaternions $\R+i\R+j\R$, Figure 2 represents $\Sto_A(y,r)$ in the following four situations: $y \in \R$, $0<\|\im(y)\|_A<r$, $\|\im(y)\|_A=r$ and $\|\im(y)\|_A>r$.
%\begin{center}
%\textit{Figure 2 (forse?)}
%\end{center}
As a consequence, the \textit{topology on $\Q_A$ induced by $\sto_A$ is strictly coarser than the euclidean one}. It is worth mentioning that $\sto_A(x,y)$ can be computed by a quite explicit formula. As usual, we denote by $\Delta_y:\Q_A \lra \Q_A$ the characteristic polynomial of $y$:
\[
\Delta_y(x):=x^2-xt(y)+n(y).
\]
Note that $\Delta_y$ takes values in the quadratic cone, since $x\in\C_I$ implies that $\Delta_y(x)\in\C_I\subset\Q_A$.
If $x=\alpha+I\beta$, $y=\xi+J\eta$, $z=\alpha+\ui\beta$ and $w=\xi+\ui\eta$ as above, then
%\footnote{Bisogna precisare che i simboli $\re$ e $\im$ verranno indifferentemente utilizzati sia per i complessi che per gli elementi di $A$.}:
\begin{align} \label{eq:delta-yx}
\Delta_y(x)=&
\, (\alpha+I\beta)^2-(\alpha+I\beta)(2 \, \xi)+(\xi^2+\eta^2)= \nonumber\\
=&
\, (\alpha+I\beta)^2-(\alpha+I\beta)(w+\overline{w})+w\overline{w}= \nonumber\\
=& \, \re(\Delta_w(z))+I \cdot \im(\Delta_w(z))
\end{align}
and hence, by (\ref{eq:E}), we have that 
\begin{equation} \label{eq:delta-yx-2}
\|\Delta_y(x)\|_A=|\Delta_w(z)|.
\end{equation}
In this way, $\sto_A$ can be expressed explicitly as follows:
\begin{equation} \label{eq:sto-explicit}
\sto_A(x,y)=\sqrt{\|\Delta_y(x)\|_A}.
\end{equation}
%By combining (\ref{eq:sto-inequality}), (\ref{eq:delta-yx-2}) and (\ref{eq:sto-explicit}), we infer that
%\begin{equation} \label{eq:inequ-ua}
%\end{equation}

\section{Abel Theorem for power and spherical series}
\label{Abel_Theorem_for_power_and_spherical_series}

\subsection{Radii of $\boldsymbol{\sigma_A}$- and of $\boldsymbol{\sto_A}$-convergence} 
Let $\mc{S}_A:=\{x \in A \, | \, \|x\|_A=1\}$. 
%The set $\mc{S}_A$ is closed and bounded in $A$. Since $A$ is finite dimensional, we infer that $\mc{S}_A$ is also compact. Thanks to this fact, 
The set $\mc{S}_A$ is compact, and the same is true for the set
\[
\mc{S}_A\cap \Q_A=\{\alpha+\beta J\in A\;|\; \alpha,\beta\in\R,\alpha^2+\beta^2=1, J\in\cS_A\}.
\]
Thanks to this fact, we can define the positive real constants $c_A$ and $C_A$ as follows:
\[
c_A:=\min_{x,z \in\mc{S}_A\cap \Q_A, y\in \mc{S}_A}\|(xy)z\|_A,\quad C_A:=\max_{x, y \in \mc{S}_A}\|xy\|_A.
\]
It follows immediately that
\begin{align} \label{eq:quasi-banach}
c_A\|x\|_A\|y\|_A &\leq \|xy\|_A \quad\forall x,y\in A \text{ such that $x\in\Q_A$ or $y$}\in\Q_A,\notag\\
\|xy\|_A &\leq C_A\|x\|_A\|y\|_A\quad\forall x,y\in A.
\end{align}
%Given $J\in\SS_A$, let $c_J$ be defined as $c_J:=\min_{x,z \in\mc{S}_A\cap\C_J, y\in \mc{S}_A}\|(xy)z\|_A$. Note that $c_J$ is positive, since every $x,z\in\mc{S}_A\cap\C_J$ is invertible. We then get
%\begin{equation} \label{eq:quasi-banachJ}
%c_J\|x\|_A\|y\|_A \|z\|_A\leq \|(xy)z\|_A \quad\forall x,z\in\C_J, y\in A.
%\end{equation}

%Given a subset $K$ of $\Q_A$, closed in $A$, let  
%$c_K:=\min_{x,z \in\mc{S}_A\cap K, y\in \mc{S}_A}\|(xy)z\|_A$. Note that $c_K$ is positive, since every $x,z\in\mc{S}_A\cap\Q_A$ is invertible. We then get
%\begin{equation} \label{eq:quasi-banachJ}
%c_K\|x\|_A\|y\|_A \|z\|_A\leq \|(xy)z\|_A \quad\forall x,z\in K, y\in A.
%\end{equation}

Given $y \in A$ and $n \in \N$, we denote by $(x-y)^{\punto n}$ the value at $x$ of the slice function on $\Q_A$ induced by the stem function $(z-y)^n$; that is, $(x-y)^{\punto n}:=\I((z-y)^n)(x)$.
Note that the slice function $(x-y)^{\punto n}$ coincides with the power $(x-y)^{* n}$ w.r.t.~the \emph{star product}  of power series (cf.~\cite{GeSto2008Mich} for the quaternionic case).

\begin{proposition} \label{prop:estimates}
Let $y \in \Q_A$. The following two statements hold:
\begin{itemize}
 \item[$(\mr{i})$] For each $x \in \Q_A$ and for each $n \in \N$, we have:
 \begin{equation} \label{eq:estimate-punto}
 \|(x-y)^{\punto n}\|_A \leq  C_A(1+C_A) \, \sigma_A(x,y)^n.
 \end{equation}
 Furthermore, it holds:
  \begin{equation} \label{eq:Abel}
   \lim_{n \rightarrow +\infty}\big(\|(x-y)^{\punto n}\|_A\big)^{1/n}=\sigma_A(x,y)\,.
  \end{equation}
 \item[$(\mr{ii})$] For each $x \in \Q_A$, we have:
  \begin{equation} \label{eq:estimate-sto}
   \|x-y\|_A \leq C_A(1+C_A) \, \big(\sto_A(x,y)+2 \, \|\im(y)\|_A\big)\,.
  \end{equation}
\end{itemize}
\end{proposition}
\begin{proof}
Let $x=\alpha+I\beta$ and $y=\xi+J\eta$ for some $\alpha,\beta,\xi,\eta \in \R$ with $\beta,\eta>0$ and $I,J \in \cS_A$. Define $z:=\alpha+\ui\beta$, $w:=\xi+\ui\eta$ and $z_J:=\alpha+J\beta$. Let $n \in \N$. By applying the representation formula for slice functions to $(x-y)^{\punto n}$ (see Proposition~6 of \cite{GhPe_AIM}), we obtain:
\[
(x-y)^{\punto n}=\frac{1}{2}\big((z_J-y)^n+(z_J^c-y)^n\big)-\frac{1}{2} \, I\big(J(z_J-y)^n-J(z_J^c-y)^n\big)\,.
\]
We recall Artin's theorem for alternative algebras: the subalgebra generated by two elements is always associative.
Since $(z_J-y)^n$ and $(z_J^c-y)^n$ belongs to $\C_J$, Artin's theorem implies that 
\begin{center}
$I\big(J(z_J-y)^n\big)=(IJ)(z_J-y)^n$ and $I\big(J(z_J^c-y)^n\big)=(IJ)(z_J^c-y)^n$.
\end{center}
In this way, we have that
\begin{equation} \label{eq:punto-representation}
(x-y)^{\punto n}=\frac{1-IJ}{2}(z_J-y)^n+\frac{1+IJ}{2}(z_J^c-y)^n.
\end{equation}
Thanks to (\ref{eq:E}), we have that $\|z_J-y\|_A=|z-w|$ and $\|z_J^c-y\|_A=|\overline{z}-w|=|z-\overline{w}|$. By combining the latter equalities with (\ref{eq:quasi-banach}) and (\ref{eq:punto-representation}), we obtain that
\begin{align} \label{eq:useful}
\|(x-y)^{\punto n}\|_A \leq & \, C_A \left\|\frac{1-IJ}{2}\right\|_A|z-w|^n+C_A\left\|\frac{1+IJ}{2}\right\|_A|z-\overline{w}|^n \leq \nonumber\\
\leq & \, C_A\frac{1+C_A}2\big(|z-w|^n+|z-\overline{w}|^n\big)
\end{align}
and hence $\|(x-y)^{\punto n}\|_A \leq C_A(1+C_A) \, \sigma_A(x,y)^n$, as desired. This proves (\ref{eq:estimate-punto}).

Let us show (\ref{eq:Abel}). If $x \in \C_J$, then $\|(x-y)^{\punto n}\|_A=\sigma_A(x,y)^n$ and hence (\ref{eq:Abel}) holds. Let $x \not\in \C_J$. Suppose $|z-w|<|z-\overline{w}|$. Since $|z-\overline{w}| \neq 0$, $z_J^c-y$ belongs to $\C_J \setminus \{0\}$. In particular, it is invertible in $\C_J$ (and hence in $A$). Define $Q:=(z_J-y)(z_J^c-y)^{-1}$. Observe that $Q \in \C_J$, $\|Q\|_A=|z-w||z-\overline{w}|^{-1}<1$ and $\|Q^n\|_A=\|Q\|_A^n$ for each $n \in \N$. In particular, it holds:
\begin{equation} \label{eq:lim}
\lim_{n \rightarrow +\infty}\|Q^n\|_A=0.
\end{equation}
By using Artin's theorem again, we obtain that
\[
\big((IJ)Q^n\big)(z_J^c-y)^n=(IJ)\big(Q^n(z_J^c-y)^n\big)=(IJ)(z_J-y)^n
\]
and hence, thanks to (\ref{eq:punto-representation}), we have: 
\begin{equation} \label{eq:Artin}
(x-y)^{\punto n}=\left(\frac{IJ+1}{2}-\frac{IJ-1}{2} \,  Q^n\right)(z_J^c-y)^n.
\end{equation}
For each $n \in \N$, define $M_n^+,M_n^- \in \R^+$ as follows:
\[
M_n^\pm:=\frac{1}{2}\big|\|IJ+1\|_A\pm \|(IJ-1)Q^n\|_A\big|.
\]
Points (\ref{eq:quasi-banach})  and (\ref{eq:Artin}) ensures that
\[
c_A M_n^-|z-\overline{w}|^n \leq \|(x-y)^{\punto n}\|_A \leq C_AM_n^+|z-\overline{w}|^n.
\]
Since $I \neq J$, $\|IJ+1\|_A \neq 0$ and hence (\ref{eq:lim}) implies that $\lim_{n \rightarrow +\infty}(M_n^\pm)^{1/n}=1$. It follows that $\lim_{n \rightarrow +\infty}\|(x-y)^{\punto n}\|^{1/n}_A=|z-\overline{w}|=\sigma_A(x,y)$. In the case in which $|z-w| > |z-\overline{w}|$, the proof of (\ref{eq:Abel}) is similar.

It remains to prove $(\mr{ii})$. By (\ref{eq:sto-inequality}), it follows that
\[
\max\{|z-w|,|z-\overline{w}|\} \leq \sqrt{|\Delta_w(z)|}+2|\im(w)|=\sto_A(x,y)+2 \, \|\im(y)\|_A.
\]
By combining inequality (\ref{eq:useful}) with $n=1$ and the preceding one, we infer at once (\ref{eq:estimate-sto}).
\end{proof}

\begin{definition}
Given $y \in \Q_A$,  we define, following \cite{StoppatoAdvMath2012}, the family of slice regular polynomial functions $\{\stx_{y,n}:\Q_A \lra A\}_{n \in \N}$:
\[
\stx_{y,2m}(x):=\Delta_y(x)^m=\I(\Delta_y(z)^m),
\quad
\stx_{y,2m+1}(x):=\Delta_y(x)^m(x-y)=\I(\Delta_y(z)^m(z-y))
\]
for each $m \in \N$.  
\end{definition}

\begin{lemma} \label{lem:slice-norm-sto}
For each $x,y \in \Q_A$, it holds:
\begin{equation} \label{eq:Abel-sto}
\lim_{n \rightarrow +\infty}\|\stx_{y,n}(x)\|_A^{1/n}=\sto_A(x,y).
\end{equation}
\end{lemma}
\begin{proof}
If $x=y$, then the statement is evident. Suppose $x \neq y$. Write $x$ and $y$ as follows: $x=\alpha+I\beta$ and $y=\xi+J\eta$ with $\alpha,\beta,\xi,\eta \in \R$ and $I,J \in \cS_A$. Define $z:=\alpha+\ui\beta$ and $w:=\xi+\ui\eta$. Let $m \in \N$. By (\ref{eq:delta-yx}), we have that
\[
\Delta_y(x)^m=\re(\Delta_w(z)^m)+I \cdot \im(\Delta_w(z)^m).
\]
Bearing in mind (\ref{eq:E}) and (\ref{eq:delta-yx-2}), we infer that
\[
\|\Delta_y(x)^m\|_A=\sqrt{{\left(\re(\Delta_w(z)^m)\right)}^2+{\left(\im(\Delta_w(z)^m)\right)}^2}=|\Delta_w(z)^m|=|\Delta_w(z)|^m=\|\Delta_y(x)\|_A^m.
\]
In particular, the limit $\lim_{m \rightarrow +\infty}(\|\stx_{y,2m}(x)\|_A\big)^{1/2m}$ is equal to $\sqrt{\|\Delta_y(x)\|_A}$, which in turn coincides with $\sto_A(x,y)$ by (\ref{eq:sto-explicit}). Moreover, by (\ref{eq:quasi-banach})  we know that
\[
c_A\|\Delta_y(x)\|_A^m\|x-y\|_A
\leq 
\|\stx_{y,2m+1}(x)\|_A
\leq
C_A\|\Delta_y(x)\|_A^m\|x-y\|_A
\]
and hence
\begin{align*}
\lim_{m \rightarrow +\infty}(\|\stx_{y,2m+1}(x)\|_A\big)^{1/(2m+1)}=& \, \lim_{m \rightarrow +\infty}\|\Delta_y(x)\|_A^{m/(2m+1)}\|x-y\|_A^{1/(2m+1)}=\\
=& \, \sqrt{\|\Delta_y(x)\|_A}=\sto_A(x,y).
\end{align*}
This completes the proof.
\end{proof}

\subsection{Abel Theorem}

If $S$ is a subset of $\Q_A$, then we denote by $\cl(S)$ the closure of $S$ in $\Q_A$.

\begin{theorem} \label{thm:radii}
Let $y \in \Q_A$, let $\{a_n\}_{n \in \N}$ be a sequence in $A$ and let $R \in \R^+ \cup \{+\infty\}$ defined by setting
\begin{equation*} %\label{eq:radius}
\limsup_{n \rightarrow +\infty}\big(\|a_n\|_A\big)^{1/n}=\frac{1}{R}.
\end{equation*}
Then the following two statements hold:
\begin{itemize}
 \item[$(\mr{i})$] The power series $\mr{P}(x)=\sum_{n \in \N}(x-y)^{\punto n}a_n$ converges totally on compact subsets of $\Q_A$ contained in $\Sigma_A(y,R)$.
If  $\{a_n\}_{n\in\N}\subset \Q_A$, then $\mr{P}$ does not converge on $\Q_A \setminus \cl\big(\Sigma_A(y,R)\big)$. We call $R$ the \emph{$\sigma_A$--radius of convergence of $\mr{P}$}.
 \item[$(\mr{ii})$] The \emph{spherical series} %\marginpar{def. spherical series} 
$\mr{S}(x)=\sum_{n \in \N}\stx_{y,n}(x)a_n$ converges totally on compact subsets of $\Q_A$ contained in $\Sto_A(y,R)$.
If  $\{a_n\}_{n\in\N}\subset \Q_A$, then $\mr{S}$ does not converge on $\Q_A \setminus \cl\big(\Sto(y,R)\big)$. We call $R$ the \emph{$\sto_A$--radius of convergence of $\mr{S}$}.
\end{itemize}
\end{theorem}
\begin{proof}
Thanks to Proposition~\ref{prop:estimates}, the proof is quite standard. We will prove the theorem only in the case in which $R \in \R^+, R\ne0$. The proof in the remaining case $R \in \{0,+\infty\}$ is similar.

Let $L$ be a compact subset of $\Q_A$ contained in $\Sigma_A(y,R)$ and let $r:=\sup_{\, x \in L}\sigma_A(x,y)$. Then $r<R$.  
By combining (\ref{eq:quasi-banach}) and (\ref{eq:estimate-punto}) with the definition of $R$, we obtain that%and (\ref{eq:radius}), we obtain that
\begin{align*}
\sum_{n \in \N}\sup_{\, x \in L}\left\|(x-y)^{\punto n}a_n\right\|_A \leq & \, C_A \sum_{n \in \N}\left(\sup_{\, x \in L}\|(x-y)^{\punto n}\|_A\right)\|a_n\|_A \leq \\
\leq & \, (C_A)^2(C_A+1)\sum_{n \in \N}r^n\|a_n\|_A<+\infty.
\end{align*}
It follows that $\mr{P}$ is totally convergent on compact subsets of $\Q_A$ contained in $\Sigma_A(y,R)$. Proceeding similarly, but using (\ref{eq:estimate-sto}) instead of (\ref{eq:estimate-punto}), we obtain also that $\mr{S}$ is totally convergent on compact subsets of $\Q_A$ contained in $\Sto_A(y,R)$.

Let now $x \in \Q_A \setminus \cl\big(\Sigma_A(y,R)\big)$ and let $s:=\sigma_A(x,y)>R$. By (\ref{eq:quasi-banach}) and (\ref{eq:Abel}), we have that
\[
\limsup_{n \rightarrow +\infty}\big(\|(x-y)^{\punto n}a_n\|_A\big)^{1/n} \geq \limsup_{n \rightarrow +\infty}\big(c_A\|(x-y)^{\punto n}\|_A\|a_n\|_A\big)^{1/n}=\frac{s}{R}>1
\]
and hence $\mr{P}(x)$ does not converge. Proceeding similarly, but using (\ref{eq:Abel-sto}) instead of (\ref{eq:Abel}), we obtain also that $\mr{S}(x)$ does not converge for each $x \in \Q_A \setminus \cl(\Sto_A(y,R))$.
\end{proof}

%\begin{remark}
 In the quaternionic case, the previous result was proved in \cite{GeSto2012MathAnn} and \cite{StoppatoAdvMath2012}. If $A$ is the Clifford algebra $\R_n$, then the quadratic cone contains the space $\R^{n+1}$ of paravectors. We then obtain the radii of convergence of power and spherical series with paravector coefficients. If not all of the coefficients $a_n$ belong to the quadratic cone, then the divergence of the series centered at $y\in\C_J$ is assured only at points $x\in  \C_J \setminus \cl\big(\Sigma_A(y,R)\big)$.
%\end{remark}

\begin{corollary}
Let $y \in \Q_A$, let $\{a_n\}_{n \in \N}$ be a sequence in $A$ and let $R \in \R^+$ defined by setting
\begin{equation*} %\label{eq:radius}
\limsup_{n \rightarrow +\infty}\big(\|a_n\|_A\big)^{1/n}=\frac{1}{R}.
\end{equation*}
Assume that $R>0$. Then the following two statements hold:
\begin{itemize}
 \item[$(\mr{i})$] If $\OO(y,R) \neq \emptyset$, then the power series $\mr{P}(x)=\sum_{n \in \N}(x-y)^{\punto n}a_n$ defines a slice regular function on $\OO(y,R)$.
 \item[$(\mr{ii})$] The spherical series $\mr{S}=\sum_{n \in \N}\stx_{y,n}a_n$ defines a slice regular function on $\Sto_A(y,R)$.
\end{itemize}
\end{corollary}
\begin{proof}
Let $y=\xi+J\eta$ with $\xi,\eta \in \R$ and $J \in \cS_A$. Let $w=\xi+i\eta$. The power series $\mr{P}(x)$ is the uniform limit on the open set $\OO(y,R)$ of the sequence of slice regular polynomials induced by the stem functions $\sum_{n=0}^N(z-y)^{n}a_n$, $N\in\N$. Similarly, the spherical series $\mr{S}(x)$ is the uniform limit of slice regular polynomials. The thesis is a consequence of Remark~\ref{sequence}.
\end{proof}

\subsection{Coefficients of power and spherical series} %: comparison and computation.} 
In the next result, %we will prove some basic properties of convergent power series. In particular, 
we will show how to compute the coefficients of a power series by means of the Cullen derivatives of the function defined by the series itself. The result is similar to the one valid in the complex variable case. We consider only series centered at points $y\in\Q_A\setminus\R$, since if $y$ is real, then a power series w.r.t.\ the slice product coincides with a standard power series.

\begin{proposition} \label{thm:taylor}
Let $y \in \Q_A \setminus \R$ and let $\mr{P}:\Sigma_A(y,R) \lra A$ be a function defined by a power series $\mr{P}(x)=\sum_{n \in \N}(x-y)^{\punto n}a_n$ centered at $y$ with positive $\sigma_A$--radius of convergence. 
Then
%$\mr{P}$ belongs to $\mscr{C}^{\oo}(\Sigma_A(y,R),A)$
%\footnote{Localmente, direi che \`e vero. Globalmente va visto bene. Ad ogni modo, a meno di definire la derivata di Cullen anche sulle funzioni aventi dominio solo su qualche $\C_J$, non ci dovrebbero essere problemi.} and, 
 it holds, for each $n \in \N$,
\[
a_n=\frac{1}{n!} \, \partial_C^n \mr{P}(y).
\]
%Furthermore, if the circular open subset $\OO(y,R) $ of $\Q_A$, defined in $(\ref{eq:OO-ball})$, is non--empty, then $f|_{\OO(y,R)} \in \mc{SR}(\OO(y,R),A)$.
\end{proposition}
\begin{proof}
%\marginpar{usare Prop. 7?}
Since $\mr{P}(x)=\lim_{N\rightarrow\infty}\I(\sum_{n=0}^N(z-y)^{n}a_n)$ and $\partial_C f(y)=\dd fx(y)$ for every slice function $f$, the result follows from $\dd {}z(z-y)^n=n(z-y)^{n-1}$ for all $n\in\N$.
\end{proof}

In order to compute the coefficients of a spherical series in terms of the slice regular function defined by the spherical series itself, we need some preparations.
%For each $a \in \N$ and for each $b \in \N \cup \{-1\}$, denote by ${a \choose b}$ the usual binomial coefficient if $0 \leq b \leq a$ and define ${a \choose b}:=0$ otherwise. 
For each $t \in \R^+$, let $\lfloor t \rfloor:=\max\{n \in \N \, | \, n \leq t\}$. %and let $\lceil t \rceil:=\min\{n \in \N \, | \, t \leq n\}$.
For each $n,\ell \in \N$, we define $e_{n\ell} \in \N$ by setting
\begin{equation*} %\label{eq:e}
e_{n\ell}:=
\left\{
 \begin{array}{ll}
  {k \choose m-k} & \mbox{ if } \; (n,\ell)=(2m,2k) \mbox{ or }(n,\ell)=(2m+1,2k) \vspace{.3em}\\
  {k \choose m-k-1} & \mbox{ if } \; (n,\ell)=(2m,2k+1) \vspace{.3em}\\
  -{k+1 \choose m-k} & \mbox{ if } \; (n,\ell)=(2m+1,2k+1),
 \end{array}
\right.
\end{equation*}
where ${a \choose b}=0$ if $b<0$ or $b>a$.
It is immediate to verify that
\begin{equation} \label{eq:range}
e_{n\ell}=0 \, \mbox{ if } \, \ell<\lfloor n/2 \rfloor \, \mbox{ or } \, \ell>n, \, \mbox{ and } \, e_{nn}=(-1)^n \, \mbox{ for each } \, n \in \N. 
\end{equation}
In particular, $\mk{e}=(e_{n\ell})_{n,\ell \in \N}$ is an infinite lower triangular matrix, which diagonal coefficients equal to $1$. We call $\mk{e}$ \textit{spherical matrix}. The structure of $\mk{e}$ is easy to visualize:
\[
\textstyle
\mk{e}=
\left(
\begin{array}{cccccccc}
1      & 0      & 0      & 0      & 0      & 0 & 0 & \cdots \\
e_{10} & -1      & 0      & 0      & 0      & 0 & 0 & \cdots \\
0      & e_{21} & 1      & 0      & 0      & 0 & 0 & \cdots \\
0      & e_{31} & e_{32} &-1      & 0      & 0 & 0 & \cdots \\
0      & 0      & e_{42} & e_{43} & 1      & 0 & 0 & \cdots \\
0      & 0      & e_{52} & e_{53} & e_{54} &-1 & 0 & \cdots %\vspace{-.7em} 
\\
0      & 0      & 0 & e_{63} & e_{64} & e_{65} & 1 & \cdots \\
\vdots & \vdots & \vdots & \vdots & \vdots & \vdots & \vdots & \ddots 
\end{array}
\right)
\]
Given $n \in \N$ and $E=(E_0,E_1,\ldots,E_n) \in A^{n+1}$, we denote by $\mk{e}_n$ the $(n+1) \times n$ matrix obtained by extracting the first $n+1$ rows and the first $n$ columns from $\mk{e}$, and by $(\mk{e}_n|E)$ the $(n+1) \times (n+1)$ matrix, whose first $n$ columns are the columns of $\mk{e}_n$ and whose last column is equal to $E$. %The matrix $(\mk{e}_n|E)$ can be written as follows:
%\[
%\textstyle
%(\mk{e}_{n}|E)=
%\left(
%\begin{array}{ccccc}
%1      & 0      & \ldots & 0         & E_0    \\
%e_{10} & 1      & \ddots & \vdots    & E_1    \\
%\vdots & \ddots & \ddots & 0         & \vdots \\
%\vdots &        & \ddots & 1         & E_{n-1}\\
%e_{n 0} & \ldots & \ldots & e_{n,n-1} & E_n   \\
%\end{array}
%\right)
%\]
The reader observes that, since the $e_{n\ell}$'s are integers, the determinant $\det (\mk{e}_n|E)$ of $(\mk{e}_n|E)$ can be defined in the usual way, without ambiguity.
%\footnote{Segue un risultato che confronta power e spherical series, e calcola i coefficienti spherical con centratura dello sviluppo fuori da $\R$. Bisognerebbe preventivamente osservare che se le series sono centrate su $\R$ allora coincidono con le usuali power series.} 

In the next result, %we will prove some basic properties of convergent power series. In particular, 
we will show how to compute the coefficients of a spherical series by means of the slice derivatives of the sum of the series itself. Observe that if $y$ is real, a spherical series coincides with a standard power series.

\begin{theorem} \label{thm:coefficients}
Let $y \in \Q_A \setminus \R$ and let $\mr{S}:\Sto_A(y,R) \lra A$ be a slice regular function defined by a spherical series $\mr{S}(x)=\sum_{n \in \N}\stx_{y,n}(x)s_n$ centered at $y$ with positive $\sto_A$--radius of convergence $R$. %Write $y$ as follows: $y=\xi+K\eta \in \Q_A \setminus \R$, where $\xi,\eta \in \R$ and $K \in \cS_A$.
The following assertions hold.
\begin{itemize}
 \item[$(\mr{i})$] For each $n \in \N$, define $E_n \in A$ by setting
 \begin{equation} \label{eq:E_n}
 E_n:=
\left\{
 \begin{array}{ll}
  \frac{1}{m!} \, (2\im(y))^m \frac{\partial^m}{\partial x^m}\mr{S}(y) & \mbox{ if } \; n=2m \vspace{.3em}\\
  \frac{1}{m!} \, (-2\im(y))^m \frac{\partial^m}{\partial x^m}\mr{S}(y^c) & \mbox{ if } \; n=2m+1.
 \end{array}
\right.
 \end{equation}
Denote by $\mk{E}$ and $\mk{s}$ the infinite vectors $(E_n)_{n \in \N}$ and $\big((2\im(y))^ns_n\big)_{n \in \N}$ in $A^{\N}$, respectively. Then $\mk{s}$ is the unique solution of the following infinite lower triangular linear system:
 \begin{equation} \label{eq:system}
 \mk{e} \cdot \mk{s}=\mk{E}.
 \end{equation}
 In other words, for each $n \in \N$, it holds:
 \begin{equation} \label{eq:system-bis}
 s_n=(-2\im(y))^{-n}\det(\mk{e}_n|\mk{E}_n),
 \end{equation}
where $\mk{E}_n$ is the column vector $(E_0,E_1,\ldots,E_n)$. In particular, the coefficients $\{s_n\}_n$ of $\mr{S}$ are uniquely determined %by the imaginary part of $y$ and 
by the derivatives $\{(\partial^n \mr{S}/\partial x^n)(y)\}_n$ and $\{(\partial^n \mr{S}/\partial x^n)(y^c)\}_n$.
  \item[$(\mr{ii})$] For each $n \in \N$, we have:
 \begin{equation} \label{eq:taylor-bis}
 \frac{1}{n!} \, \frac{\partial^n}{\partial x^n}\mr{S}(y)=(2\im(y))^{-n}\sum_{\ell=n}^{2n}e_{2n,\ell}(2\im(y))^{\ell}s_{\ell}.
 \end{equation}
% \item[$(\mr{ii}^{\pr})$] Let $J \in \cS_A$ such that $y \in \C_J$ and let $\mr{P}(x)=\sum_{n \in \N}(x-y)^{\punto n}a_n$ be a power series centered at $y$ with positive $\sigma_A$--radius of convergence such that $\mr{S}$ and $\mr{P}$ coincides on some neighborhood of $y$ in $\C_J$. Then, for each $n \in \N$, we have:
%\begin{equation} \label{eq:taylor} a_n=(2\im(y))^{-n}\sum_{\ell=n}^{2n}e_{2n,\ell}(2\im(y))^{\ell}s_{\ell}.
% \end{equation}
\end{itemize}
\end{theorem}

In order to prove this result, we need two preliminary lemmas.
For a fixed $y \in \C \setminus \R$ and $z\in\C$, let $\stx_{y,2m}(z):=\Delta_y(z)^m$ and $\stx_{y,2m+1}(z):=\Delta_y(z)^m(z-y)$.

\begin{lemma} \label{lem:partial}
%Let $y \in \C \setminus \R$. Then, for 
For each $n,\ell \in \N$, it holds:
\begin{equation} \label{eq:uno}
\frac{\partial^n}{\partial z^n}\stx_{y,\ell}(y)=n! \, e_{2n,\ell} (2\im(y))^{\ell-n}
\end{equation}
and
\begin{equation} \label{eq:due}
\frac{\partial^n}{\partial z^n}\stx_{y,\ell}(\overline{y})=n! \, (-1)^\ell e_{2n+1,\ell} (-2\im(y))^{\ell-n}
\end{equation}
\end{lemma}
\begin{proof}
First, suppose $\ell=2k$ for some $k \in \N$. Since $\stx_{y,2k}(z)=(z-y)^k(z-\overline{y})^k$, it holds:
\begin{align*}
& \frac{\partial^n}{\partial z^n}\stx_{y,\ell}(z)=
\sum_{m=0}^n{n \choose m}\frac{\partial^m}{\partial z^m}\big((z-y)^k\big) \, \frac{\partial^{n-m}}{\partial z^{n-m}}\big((z-\overline{y})^k\big)=\\
& =
\sum_{m=0}^n{n \choose m}k(k-1)\cdots(k-m+1)(z-y)^{k-m}k(k-1)\cdots(k-n+m+1)(z-\overline{y})^{k-n+m}.
\end{align*}
It follows that, if $(\partial^n \stx_{y,\ell}/\partial z^n)(y) \neq 0$, then $k \leq n$ and $k-n+k \geq 0$; that is, $k \leq n \leq 2k$. Furthermore, if $k \leq n \leq 2k$, then
\[
\frac{\partial^n}{\partial z^n}\stx_{y,\ell}(y)={n \choose k}k! \cdot k(k-1)\cdots(2k-n+1)(y-\overline{y})^{2k-n}=n! \, {k \choose n-k}(2\im(y))^{2k-n}.
\]
This proves (\ref{eq:uno}) if $\ell$ is even. Similarly, one shows (\ref{eq:due}) if $\ell$ is even.

Let now $\ell=2k+1$ for some $k \in \N$. Since $\stx_{y,2k+1}(z)=(z-y)\stx_{y,2k}(z)$, we have:
\[
\frac{\partial^n}{\partial z^n}\stx_{y,2k+1}(z)=(z-y) \, \frac{\partial^n}{\partial z^n}\stx_{y,2k}(z)+n \, \frac{\partial^{n-1}}{\partial z^{n-1}}\stx_{y,2k}(z).
\]
By combining the latter fact with the preceding part of the proof, we obtain:
\[
\frac{\partial^n}{\partial z^n}\stx_{y,2k+1}(y)=n! \, {k \choose n-k-1}(2\im(y))^{2k+1-n}\quad \text{and}
\]
%and
\begin{align*}
\frac{\partial^n}{\partial z^n}\stx_{y,2k+1}(\overline{y})&= (\overline{y}-y) \, n! \, {k \choose n-k}(-2\im(y))^{2k-n}+\\
&+n! \, {k \choose n-k-1}(-2\im(y))^{2k+1-n}=%\\&= \,
n! \, {k+1 \choose n-k}(-2\im(y))^{2k+1-n},
\end{align*}
as desired.
\end{proof}

The next lemma is a complex version of Theorem \ref{thm:coefficients}.

\begin{lemma} \label{lem:second}
Let $y \in \C \setminus \R$ and let $\mr{S}:\Sto(y,R) \lra \C$ be a holomorphic function defined by a complex spherical series $\mr{S}(z)=\sum_{n \in \N}\stx_{y,n}(z)s_n$ centered at $y$ with positive $\sto$--radius of convergence $R$. %Write $y$ as follows: $y=\xi+K\eta \in \Q_A \setminus \R$, where $\xi,\eta \in \R$ and $K \in \cS_A$.
The following assertions hold.
\begin{itemize}
 \item[$(\mr{i})$] For each $n \in \N$, define $E_n \in \C$ by formula $(\ref{eq:E_n})$, where $\partial^m/\partial x^m$ must be replaced by $\partial^m/\partial z^m$ and $y^c$ by $\overline{y}$. Moreover, define the infinite vectors $\mk{E}$ and $\mk{s}$ in $\C^{\N}$ by setting $\mk{E}:=(E_n)_{n \in \N}$ and $\mk{s}:=\big((2\im(y))^ns_n\big)_{n \in \N}$. Then $\mk{s}$ is the unique solution of the infinite lower triangular linear system $ \mk{e} \cdot \mk{s}=\mk{E}$. %More explicitly, for each $n \in \N$, it holds:
%\begin{equation} \label{eq:system-tris}
%s_n=(-2\im(y))^{-n}\det(e_n|\mk{E}_n),
%\end{equation}
%where $\mk{E}_n$ is the column vector $(E_0,E_1,\ldots,E_n)$ in $\C^{n+1}$.
  \item[$(\mr{ii})$] For each $n \in \N$, formula $(\ref{eq:taylor-bis})$ holds.
\end{itemize}
\end{lemma} 
\begin{proof}
Let $n \in \N$. By (\ref{eq:range}) and (\ref{eq:uno}), we have the equality
\begin{equation*} %\label{eq:bridge}
%\frac{\partial^n}{\partial z^n}\mr{S}(y)=\sum_{\ell=n}^{2n}n! \, e_{2n,\ell}(2\im(y))^{\ell-n}s_{\ell}=n! \, (2\im(y))^{-n}\sum_{\ell=n}^{2n}e_{2n,\ell}(2\im(y))^{\ell}s_{\ell}.
\frac{\partial^n}{\partial z^n}\mr{S}(y)=\sum_{\ell \in \N}n! \, e_{2n,\ell}(2\im(y))^{\ell-n}s_{\ell}=n! \, (2\im(y))^{-n}\sum_{\ell \in \N} e_{2n,\ell}(2\im(y))^{\ell}s_{\ell},
\end{equation*}
which implies immediately $(\mr{ii})$. The latter equality is equivalent to the following one:
\begin{equation} \label{eq:tre}
%E_{2n}=\sum_{\ell=n}^{2n}e_{2n,\ell}(2\im(y))^{\ell}s_{\ell},
E_{2n}=\sum_{\ell \in \N}e_{2n,\ell}(2\im(y))^{\ell}s_{\ell},
\end{equation}
because $E_{2n}=(n!)^{-1}(2\im(y))^n(\partial^n \mr{S}/\partial z^n)(y)$ by definition. 
%Write $\mr{S}$ as follows:
%\[
%\mr{S}(z)=\sum_{k \in \N}\stx_{y,2k}(z)s_{2k}+\sum_{k \in \N}\stx_{y,2k+1}(z)s_{2k+1}.
%\]
Thanks to (\ref{eq:range}) and (\ref{eq:due}), we obtain:
\begin{align*}
\frac{\partial^n}{\partial z^n}\mr{S}(\overline{y})&=
\sum_{\ell \in \N}n! \, (-1)^\ell e_{2n+1,\ell}(-2\im(y))^{\ell-n}s_{\ell}=\\
%\sum_{k \in \N} n! \, e_{2n+1,2k}(-2\im(y))^{2k-n}s_{2k}+\\
%& \, +\sum_{k \in \N} n! \, e_{2n+1,2k+1}(-2\im(y))^{2k+1-n}s_{2k+1}=\\
%=& \,
&=n! \, (-2\im(y))^{-n}\sum_{\ell \in \N} e_{2n+1,\ell}(2\im(y))^{\ell}s_{\ell}.
\end{align*}
On the other hand, by definition, $E_{2n+1}=(n!)^{-1}(-2\im(y))^n(\partial^n \mr{S}/\partial z^n)(\overline{y})$ and hence we have:
\begin{equation} \label{eq:quattro}
%E_{2n+1}=\sum_{\ell=n}^{2n+1}e_{2n+1,\ell}s_{\ell}.
E_{2n+1}=\sum_{\ell \in \N}e_{2n+1,\ell}(2\im(y))^{\ell}s_{\ell}.
\end{equation}
Evidently, the infinite linear system formed by equations (\ref{eq:tre}) and (\ref{eq:quattro}) for every $n \in \N$ coincides with $\mk{e} \cdot \mk{s}=\mk{E}$. %The equivalence between the latter linear system and equations (\ref{eq:system-tris}) follows immediately from Cramer's role.
This proves $(\mr{i})$. 
\end{proof}

We are now in position to prove Theorem \ref{thm:coefficients}.

\begin{proof}[Proof of Theorem \ref{thm:coefficients}]
Let $\xi,\eta \in \R$ and let $J \in \cS_A$ such that $y=\xi+J\eta$. Since $y \not\in \R$, $\eta \neq 0$. Define $w:=\xi+\ui\eta \in D$. According to (\ref{eq:defs}), denote by $\Phi_J:\C \lra \Q_A$ and $\mr{S}_J:D \lra A$ the functions defined by $\Phi_J(\alpha+\ui\beta):=\alpha+J\beta$ and $\mr{S}_J(z):=\mr{S}(\Phi_J(z))$. Let $n \in \N$. Observe that, given $z=\alpha+\ui\beta \in \C$, we have:
\begin{align} \label{eq:phi_j}
\stx_{y,n}(\Phi_J(z))=& \, (\alpha+J\beta)^2-(\alpha+J\beta)(2\xi)+\xi^2+\eta^2= \nonumber \\
=& \,
\re(\stx_{w,n}(z))+J\im(\stx_{w,n}(z))=\Phi_J(\stx_{w,n}(z)).
\end{align}
Choose a splitting base $(1,J,J_1,JJ_1,\ldots,J_h,JJ_h)$ of $A$ associated with $J$ and denote by $s_{n,1,0},s_{n,2,0},\ldots,s_{n,1,h},s_{n,2,h}$ the real numbers such that $s_n=\sum_{k=0}^h(s_{n,1,k}+Js_{n,2,k})J_k$, where $J_0:=1$. Define $\hat{s}_{n,k} \in \C$ by setting $\hat{s}_{n,k}:=s_{n,1,k}+\ui s_{n,2,k}$. Evidently, it holds:
\begin{equation} \label{eq:s_n}
\textstyle
s_n=\sum_{k=0}^h\Phi_J(\hat{s}_{n,k}) \, J_k.
\end{equation}
By combining the latter equality with (\ref{eq:phi_j}), we infer that
\begin{align*}
\mr{S}_J(z)=& \,
%\textstyle
\sum_{n \in \N}\Phi_J(\stx_{w,n}(z))\left(\sum_{k=0}^h\Phi_J(\hat{s}_{n,k}) \, J_k\right)= \nonumber 
\\=& \,
%\textstyle
\sum_{k=0}^h\left(\sum_{n \in \N}\Phi_J(\stx_{w,n}(z))\Phi_J(\hat{s}_{n,k})\right)J_k= \nonumber 
%\\=& \,
%\textstyle
\sum_{k=0}^h\Phi_J\left(\sum_{n \in \N}\stx_{w,n}(z)\hat{s}_{n,k}\right)J_k
\end{align*}
for each $z \in \Sto(w,R)$. In particular, it follows that, for each $k \in \{0,1,\ldots,h\}$, the complex spherical series  $\hat{\mr{S}}_k:=\sum_{n \in \N}\stx_{w,n}\hat{s}_{n,k}$ converges on $\Sto(w,R)$ and hence its $\sto$--radius of convergence is $\geq R$. For each $k \in \{0,1,\ldots,h\}$, define $\hat{E}_{n,k} \in \C$ as follows:
\begin{equation*} %\label{eq:E_n}
\hat{E}_{n,k}:=
\left\{
 \begin{array}{ll}
  \frac{1}{m!} \, (2\ui\eta)^m \frac{\partial^m}{\partial x^m}\hat{\mr{S}}_k(w) & \mbox{ if } \; n=2m \vspace{.3em}\\
  \frac{1}{m!} \, (-2\ui\eta)^m \frac{\partial^m}{\partial x^m}\hat{\mr{S}}_k(\overline{w}) & \mbox{ if } \; n=2m+1.
 \end{array}
\right.
\end{equation*}
Moreover, define $\hat{\mk{E}}_k$ and $\hat{\mk{s}}_k$ in $\C^{\N}$ by setting
\[
\hat{\mk{E}}_k:=(\hat{E}_{n,k})_{n \in \N}
\; \; \mbox{ and } \; \;
\hat{\mk{s}}_k:=((-2\ui\eta)^n\hat{s}_{n,k})_{n \in \N}.
\]
By point $(\mr{i})$ of Lemma \ref{lem:second}, we know that, for each $k \in \{0,1,\ldots,h\}$, $\hat{\mk{s}}_k$ is the unique solution of the lynear system $\mk{e} \cdot \hat{\mk{s}}_k=\hat{\mk{E}}_k$. This is equivalent to assert that
\begin{equation} \label{eq:kn}
\sum_{\ell=\lfloor n/2 \rfloor}^ne_{n\ell}(-2\ui\eta)^{\ell}\hat{s}_{\ell, k}=\hat{E}_{n,k} \; \mbox{ for each }k,n\in \N.
\end{equation}
By Lemma \ref{lem:cullen}, we know that
\begin{equation} \label{eq:dd}
\dd{^n}{x^n}\mr{S}(y)=\dd{^n}{z^n}\mr{S}_J(w)=\sum_{k=0}^h\Phi_J\left(\dd{^n}{z^n}\hat{S}_k(w)\right)J_k
\end{equation}
for each $n \in \N$. In this way, it follows immediately that
\begin{equation} \label{eq:EE_n}
\textstyle
E_n=\sum_{k=0}^h\Phi_J(\hat{E}_{n,k})J_k.
\end{equation}
By using (\ref{eq:s_n}), (\ref{eq:kn}), (\ref{eq:EE_n}) and Artin's theorem, we obtain:
\begin{align*}
%\textstyle
\sum_{\ell=\lfloor n/2 \rfloor}^ne_{n\ell}(-2J\eta)^{\ell}s_{\ell}=& \,
%\textstyle
\sum_{\ell=\lfloor n/2 \rfloor}^ne_{n\ell}(-2J\eta)^{\ell}\left(\sum_{k=0}^h\Phi_J(\hat{s}_{\ell,k})J_k\right)=\\
=& \,
%\textstyle
\sum_{\ell=\lfloor n/2 \rfloor}^n\sum_{k=0}^h\Phi_J\!\left(e_{n\ell}(-2\ui\eta)^{\ell}\hat{s}_{\ell,k}\right)\!J_k=\\
=& \,
%\textstyle
\sum_{k=0}^h\Phi_J\!\!\left(\sum_{\ell=\lfloor n/2 \rfloor}^ne_{n\ell}(-2\ui\eta)^{\ell}\hat{s}_{\ell,k}\right)J_k=\\
=& \,
%\textstyle
\sum_{k=0}^h\Phi_J(\hat{E}_{n,k})J_k=E_n
\end{align*}
for each $n \in \N$. In other words, $\mk{s}$ is the unique of the linear system (\ref{eq:system}), as desired. The equivalence between linear system (\ref{eq:system}) and equations (\ref{eq:system-bis}) follows immediately from Cramer's rule. This proves $(\mr{i})$.

Let us prove $(\mr{ii})$. By point $(\mr{ii})$ of Lemma \ref{lem:second}, we know that
\begin{equation*} %\label{eq:taylor-tris}
\frac{1}{n!} \, \frac{\partial^n}{\partial z^n}\hat{\mr{S}}_k(w)=(2\ui\eta)^{-n}\sum_{\ell=n}^{2n}e_{2n,\ell}(2\ui\eta)^{\ell}\hat{s}_{\ell k}
\end{equation*}
for each $k,n \in \N$. Since $\mr{S}_J=\sum_{k=0}^h\Phi_J(\hat{\mr{S}}_k)J_k$, we obtain that
\begin{align*}
\frac{1}{n!} \, \frac{\partial^n}{\partial z^n}\mr{S}_J(w)=& \,
\sum_{k=0}^h\left((2J\eta)^{-n}\sum_{\ell=n}^{2n}e_{2n,\ell}(2J\eta)^{\ell}\Phi_J(\hat{s}_{\ell k})\right)J_k=\\
=& \,
(2J\eta)^{-n}\sum_{k=0}^h\sum_{\ell=n}^{2n}e_{2n,\ell}(2J\eta)^{\ell}\Phi_J(\hat{s}_{\ell k})J_k=\\
%=& \,
%(2J\eta)^{-n}\sum_{\ell=n}^{2n}\sum_{k=0}^he_{2n,\ell}(2J\eta)^{\ell}\Phi_J(\hat{s}_{\ell k})J_k=\\
=& \,
(2J\eta)^{-n}\sum_{\ell=n}^{2n}e_{2n,\ell}(2J\eta)^{\ell}\sum_{k=0}^h\Phi_J(\hat{s}_{\ell k})J_k.
\end{align*}
By Lemma \ref{lem:cullen} and (\ref{eq:s_n}), we infer that
\begin{align*}
\frac{1}{n!} \, \frac{\partial^n}{\partial x^n}\mr{S}(y)=& \,
\frac{1}{n!} \, \frac{\partial^n}{\partial z^n}\mr{S}_J(w)=(2J\eta)^{-n}\sum_{\ell=n}^{2n}e_{2n,\ell}(2J\eta)^{\ell}\sum_{k=0}^h\Phi_J(\hat{s}_{\ell k})J_k=\\
=& \,
(2J\eta)^{-n}\sum_{\ell=n}^{2n}e_{2n,\ell}(2J\eta)^{\ell}s_{\ell}.
\end{align*}
Point (\ref{eq:taylor-bis}) is proved.
%It remains to show $(\mr{ii^{\pr}})$. .......
\end{proof}

%\marginpar{derivata sferica}
\begin{remark}
The coefficient $s_1$ in the spherical series $\mr{S}(x)=\sum_{n \in \N}\stx_{y,n}(x)s_n$ is the \emph{spherical derivative} $\partial_s \mr{S}(y)$ of $\mr{S}(x)$ at $y$ (cf.~\cite[Def.~6]{GhPe_AIM}).
\end{remark}

%\noindent \textbf{Power expansion for slice regular functions.} We begin recalling a classical result on holomorphic functions, which  goes back to Cauchy. %Given a point of the domain of the function, it asserts that the Taylor expansion converges to the function locally

%\begin{lemma} \label{lem:holomorphic}
%Let $E$ be a non--empty bounded open subset of $\C$, let $g:E \lra \C$ be a holomorphic function, let $w \in E$ and let $r \in \R^+$ such that $\bar{B}(w,r) \subset E$. Define $G:=\sup_{\partial B(w,r)}|g|$. %Suppose that $g$ admits a continuous extension on the closure of $E$ in $\C$.
%Then the following assertions are verified.
%\begin{itemize}
% \item[$(\mr{i})$] $g$ expands as follows:
% \[
% g(z)=\sum_{n \in \N}(z-w)^n \, \frac{1}{n!} \, \frac{d^ng}{dz^n}(w) \; \mbox{ for each }z \in B(w,r).
% \]
% \item[$(\mr{ii})$] For each $n \in \N$, it holds:
% \[
% \frac{1}{n!} \, \frac{d^ng}{dz^n}(w)=\frac{1}{2\pi\ui}\int_{\partial B(w,r)}g(\xi) \, (\xi-w)^{-n-1} \, d\xi
% \]
% and hence
% \[
% \frac{1}{n!} \, \left|\frac{d^ng}{dz^n}(w)\right| \leq \frac{G}{r^n}.
% \]
% \item[$(\mr{iii})$] For each $n \in \N$ and for each $z \in B(w,r)$, it holds:
% \[
% g(z)-\sum_{k=0}^n(z-w)^k \, \frac{1}{k!} \, \frac{d^kg}{dz^k}(w)=\frac{1}{2\pi\ui} \, (z-w)^{n+1}\int_{\partial B(w,r)}g(\xi) \, (\xi-w)^{-n-1}(\xi-z)^{-1} \, d\xi
% \]
% and hence
% \[
% \left| \, g(z)-\sum_{k=0}^n(z-w)^k \, \frac{1}{k!} \, \frac{d^kg}{dz^k}(w)\right| \leq G \cdot \left(\frac{|z-w|}{r}\right)^n \cdot \frac{|z-w|}{r-|z-w|}.
% \]
%\end{itemize}
%\end{lemma}
%
%A proof of this result can be found in ....... .

\section{Power expansion for slice regular functions}\label{Power_expansion_for_slice_regular_functions}

A classical result on holomorphic functions, %which  goes back to Cauchy, 
asserts that the Taylor expansion converges to the function locally. The Cauchy integral formula gives also an exact expression for the Taylor remainder.
It is a nice occurrence that such a result of fundamental importance in the context of holomorphic functions can be restated for  slice regular functions in a quite similar form.

\begin{theorem} \label{thm:power-expansion}
Let $f \in \mc{SR}(\OO_D,A)$, let $y \in \OO_D$, let $J \in \cS_A$ such that $y \in \C_J$ and let $r \in \R^+$ such that $\cl(\Sigma_A(y,r)) \subset \OO_D$. Then the following assertions hold.
\begin{itemize}
 \item[$(\mr{i})$] $f$ expands as follows:
 \[
 f(x)=\sum_{n \in \N}(x-y)^{\punto n} \, \frac{1}{n!} \, \dd{^nf}{x^n}(y) \; \mbox{ for each }x \in 
 \Sigma_A(y,r).
 \]
 \item[$(\mr{ii})$] For each $n \in \N$, it holds:
 \[
 \frac{1}{n!} \, \dd{^nf}{x^n}(y)=(2\pi J)^{-1}\int_{\partial B_J(y,r)}(\zeta-y)^{-n-1} \, d\xi \, f(\zeta).
 \]
 Furthermore, %if $f$ admits a continuous extension on $\cl(\OO_D)$, then
 there exists a positive real constant $\mr{C}$, depending only on $\mk{a}$ and on $\| \cdot \|_A$, such that
 \[
 \frac{1}{n!} \, \left\|\dd{^nf}{x^n}(y)\right\|_A \leq \mr{C}\, r^{-n}\, {\textstyle \sup_{\partial B_J(y,r)}\|f\|_A}\mbox{ for each }n \in \N.
 \]
% \item[$(\mr{iii})$] For each $n \in \N$, it holds:
% \[
% f(x)-\sum_{k=0}^n(x-y)^k \, \frac{1}{k!} \, \dd{^kf}{x^k}(y)=(2\pi J)^{-1} \, (x-y)^{n+1}\!\int_{\partial B_J(y,r)} (\xi-x)^{-1}(\xi-y)^{-n-1} d\xi \, f(\xi)
% \]
% for each $x \in B_J(y,r)$\footnote{Manca la stima corrispondente}.
 \item[$(\mr{iii})$] Suppose $\OO(y,r) \neq \emptyset$. For each $n \in \N$, it holds:
 \[
 f(x)-\sum_{k=0}^n(x-y)^{\punto k} \, \frac{1}{k!} \, \dd{^kf}{x^k}(y)=(2\pi)^{-1} \, (x-y)^{\punto n+1} \punto \mr{R}_{y,n}(f)(x)
%\footnote{C'\`e un abuso nella definizione di $(x-y)^{\punto n+1} \punto \mr{R}_{y,n}(f)(x)$, che dovrebbe essere $((\cdot-y)^{\punto n+1} \punto \mr{R}_{y,n}(f))(x)$.}
 \]
 for each $x \in \OO(y,r)$, where $\mr{R}_{y,n}(f):\OO(y,r) \lra A$ is the slice regular function induced by the holomorphic stem function
 \[
 %D \ni \alpha+\ui\beta \longmapsto \int_{\partial B_J(y,r)} (\xi-\alpha-\UI\beta)^{-1}J^{-1}(\xi-y)^{-n-1} \, d\xi \, f(\xi) \in A \otimes \C.
 %B(w,r) \cap B(\overline{w},r) \ni z \longmapsto \int_{\partial B_J(y,r)} (\xi-\zeta)^{-1}(\xi-y)^{-n-1} J^{-1} \, d\xi \, f(\xi) \in A \otimes \C,
 B(w,r) \cap B(\overline{w},r) \ni z \longmapsto \int_{\partial B_J(y,r)} \Delta_{\zeta}(z)^{-1}(\zeta^c-z)J^{-1}(\zeta-y)^{-n-1} \, d\zeta \, f(\zeta), %\in A \otimes \C,
 \]
% where $w=\xi+\ui\eta \in D$ if $y=\xi+J\eta$ and $\zeta=\alpha+\UI\beta \in \R \otimes \C$ if $z=\alpha+\ui\beta$.
 where $w=\xi+\ui\eta \in D$ if $y=\xi+J\eta$. %and $\zeta=\alpha+i\beta \in \R \otimes \C$ if $z=\alpha+\ui\beta$.
 
 Furthermore, if $A$ is associative, then  for each $x \in \OO(y,r)$ we have
 \[
 \mr{R}_{y,n}(f)(x)=\int_{\partial B_J(y,r)} \Delta_{\zeta}(x)^{-1}(\zeta^c-x)J^{-1}(\zeta-y)^{-n-1} \, d\zeta \, f(\zeta).
 \]
 \item[$(\mr{iii^{\pr}})$] Suppose $\OO(y,r) \neq \emptyset$ and define $\mc{F}_J(y,r):=\partial B_J(y,r) \cup \partial B_J(y^c,r)$. Then there exists a positive real constant $\mr{C^{\pr}}$, depending only on $\mk{a}$ and on $\| \cdot \|_A$, such that
 \[
 \left\| \, f(x)-\sum_{k=0}^n(x-y)^{\punto k} \, \frac{1}{k!} \, \dd{^kf}{x^k}(y)\right\|_A \leq \mr{C^{\pr}}   \left(\sup_{\mc{F}_J(y,r)}\|f\|_A\right)  \left(\frac{\sigma_A(x,y)}{r}\right)^n  \frac{\sigma_A(x,y)}{r-\sigma_A(x,y)}
 \]
 for each $x \in \OO(y,r)$.
\end{itemize}
\end{theorem}
\begin{proof}
By Lemma \ref{lem:universal-norm}, there exist a positive real constant $\mr{H}$, depending only on $\mk{a}$ and on $\| \cdot \|_A$, and a splitting base $\mscr{B}=(1,J,J_1,JJ_1,\ldots,J_h,JJ_h)$ of $A$ associated with $J$ such that
\begin{equation} \label{eq:2}
\|J_{\ell}\|_A=1 \; \mbox{ for each }\ell \in \{1\ldots,h\}\text{ and }\|x\|_{\mscr{B}} \leq \mr{H} \, \|x\|_A \; \mbox{ for each }x \in A.
\end{equation}
%and
%\begin{equation} \label{eq:2}
%\|x\|_{\mscr{B}} \leq \mr{H} \, \|x\|_A \; \mbox{ for each }x \in A.
%\end{equation}
Let $f_{1,0},f_{2,0},\ldots,f_{1,h},f_{2,h}$ be the functions in $\mscr{C}^1(D,\R)$ such that $f=\sum_{\ell=0}^h(f_{1,\ell}J_{\ell}+f_{2,\ell}JJ_{\ell})$, where $J_0:=1$. For each $\ell \in \{0,1,\ldots,h\}$, define the functions $\hat{f}_{\ell}:D \lra \C$ and $\widetilde{f}_{\ell}:D \lra \OO_D \cap \C_J$ by setting $\hat{f}_{\ell}:=f_{1\ell}+\ui f_{2,\ell}$ and $\widetilde{f}_{\ell}:=\Phi_J \circ \hat{f}_{\ell}$. Lemma \ref{lem:splitting} ensures that each function $\hat{f}_{\ell}$ is holomorphic. By definition of $f_J$ (see (\ref{eq:defs})), we have that
\begin{equation} \label{eq:easy}
\textstyle
f_J=\sum_{\ell=0}^h\widetilde{f}_{\ell}J_{\ell}=\sum_{\ell=0}^h(\Phi_J \circ \hat{f}_{\ell})J_{\ell} \; \mbox{ on }D.
\end{equation}
Let $\theta,\eta \in \R$ such that $y=\theta+J\eta$ and let $w:=\theta+\ui\eta \in D$. By combining (\ref{eq:easy}) with Lemma \ref{lem:cullen}, we infer that
\begin{equation} \label{eq:easy-2}
\dd{^nf}{x^n}(y)=\sum_{\ell=0}^h\Phi_J\!\! \left(\frac{d^n\hat{f}_{\ell}}{dz^n}(w)\right) J_{\ell} \; \mbox{ for each }n \in \N.
\end{equation}
%In fact, for each $n \in \N$, it holds:
%\begin{align*}
%\sum_{\ell=0}^h\Phi_J\!\! \left(\frac{d^n\hat{f}_{\ell}}{dz^n}(w)\right) J_{\ell}=&
%\,
%\sum_{\ell=0}^h\dd{^n\widetilde{f}_{\ell}}{z^n}(w) \, J_{\ell}=\dd{^n}{z^n}\left(\sum_{\ell=0}^h\widetilde{f}_{\ell}J_{\ell}\right)(w)=\\
%=& \,
%\dd{^nf_J}{z^n}(w)=\dd{^nf}{x^n}(y).
%\end{align*}
Let us now show that
\begin{equation} \label{eq:3}
\textstyle
\sum_{\ell=0}^h\sup_{\partial B(w,r)}|\hat{f}_{\ell}| \leq (h+1) \, \mr{H} \left(\sup_{\partial B_J(y,r)}\|f\|_A\right).
\end{equation}
In fact, bearing in mind (\ref{eq:2}), we have:
\begin{align*}
\textstyle
\sum_{\ell=0}^h\sup_{\partial B(w,r)}|\hat{f}_{\ell}| \leq&
\,
\textstyle
(h+1)\left(\sup_{\partial B(w,r)}\left|\big(f_{1,0}^2+f_{2,0}^2+\ldots+f_{1,h}^2+f_{2,h}^2\big)^{1/2}\right|\right)=\\
=&
\,
\textstyle
(h+1)\left(\sup_{\partial B_J(y,r)}\|f\|_{\mscr{B}}\right) \leq (h+1) \, \mr{H}\left(\sup_{\partial B_J(y,r)}\|f\|_A\right).
\end{align*}
Fix $x=\alpha+I\beta \in \Sigma_A(y,r)$ with $\alpha,\beta \in \R$ and $I \in \cS_A$ and define $z:=\alpha+\ui\beta$.

We are ready to prove $(\mr{i})$. By expanding each holomorphic function $\hat{f}_{\ell}$ in power series and by using (\ref{eq:easy}), (\ref{eq:easy-2}) and Artin's theorem, if $x \in B_J(y,r)$, we obtain:
\begin{align*}
f(x)&= \, f_J(z)=\sum_{\ell=0}^h\Phi_J\!\!\left(\sum_{n \in \N}(z-w)^n \, \frac{1}{n!} \, \frac{d^n\hat{f}_{\ell}}{dz^n}(w)\right)J_{\ell}=\\
&= \,
\sum_{\ell=0}^h\sum_{n \in \N}(x-y)^n \cdot \frac{1}{n!} \cdot \Phi_J\!\left(\frac{d^n\hat{f}_{\ell}}{dz^n}(w)\right) \cdot J_{\ell}=\\
&= \,
\sum_{n \in \N}(x-y)^n \, \frac{1}{n!} \left(\sum_{\ell=0}^h \Phi_J\!\left(\frac{d^n\hat{f}_{\ell}}{dz^n}(w)\right) J_{\ell}\right)=\\
&= \,
\sum_{n \in \N}(x-y)^n \, \frac{1}{n!} \, \dd{^nf}{x^n}(y)
=\sum_{n \in \N}(x-y)^{\punto n} \, \frac{1}{n!} \, \dd{^nf}{x^n}(y).
\end{align*}
If $\OO(y,r) \neq \emptyset$, then the preceding chain of equalities ensures that $f$ and the  slice regular function on $\OO(y,r)$, sending $x$ into $\sum_{n \in \N}(x-y)^{\punto n} (1/n!)(\partial^nf/\partial x^n)(y)$, coincide on $\OO(y,r) \cap \C_J$. By the general version of representation formula for slice functions given in \cite[Prop.~6]{GhPe_AIM}, %(see equality (2)), 
we infer that they coincide on the whole $\OO(y,r)$.

Fix $n \in \N$. Let us show $(\mr{ii})$. By applying Cauchy formula for derivatives to each $\hat{f}_{\ell}$ and by using (\ref{eq:easy}), (\ref{eq:easy-2}), Artin's theorem, (\ref{eq:2}) and (\ref{eq:3}), we obtain:
\begin{align*}
\frac{1}{n!} \, \dd{^nf}{x^n}(y)&= \, \sum_{\ell=0}^h\Phi_J\!\!\left(\frac{1}{n!} \, \frac{d^n\hat{f}_{\ell}}{dz^n}(w)\right)J_{\ell}=\\
&=
\sum_{\ell=0}^h\left((2\pi J)^{-1}\int_{\partial B_J(y,r)}(\zeta-y)^{-n-1} d\zeta \, \widetilde{f}_{\ell}(\Phi_J^{-1}(\zeta))\right)J_{\ell}=\\
&=
(2\pi J)^{-1}\int_{\partial B_J(y,r)}(\zeta-y)^{-n-1} d\zeta \left(\sum_{\ell=0}^h\widetilde{f}_{\ell}(\Phi_J^{-1}(\zeta))J_{\ell}\right)=\\
&=
(2\pi J)^{-1}\int_{\partial B_J(y,r)}(\zeta-y)^{-n-1} d\zeta f(\zeta)
\end{align*}
and
\begin{align*}
\frac{1}{n!} \, \left\|\dd{^nf}{x^n}(y)\right\|_A &\leq  \, C_A \sum_{\ell=0}^h\frac{1}{n!} \,\left\| \Phi_J\!\!\left(\frac{d^n\hat{f}_{\ell}}{dz^n}(w)\right)\right\|_A\|J_{\ell}\|_A=\\
&=
C_A \sum_{\ell=0}^h\frac{1}{n!} \,\left|\frac{d^n\hat{f}_{\ell}}{dz^n}(w)\right| \leq C_A \, \frac{1}{r^n} \sum_{\ell=0}^h\sup_{\partial B(w,r)}|\hat{f}_{\ell}| \leq\\
&\leq 
\,
C_A (h+1) \, \mr{H}\left({\textstyle \sup_{\partial B_J(y,r)}\|f\|_A}\right)\frac{1}{r^n} 
\end{align*}
Defining $\mr{C}:=C_A (h+1) \, \mr{H}$, we complete the proof of $(\mr{ii})$.

%The proof of $(\mr{iii})$ is similar to the one of the first part of $(\mr{ii})$.
Define the function $g:\Sigma_A(y,r) \lra A$ and, for each $\ell \in \{0,1,\ldots,h\}$, the function $\hat{g}_{\ell}:B(w,r) \lra \C$ by setting
\[
g(x):=f(x)-\sum_{k=0}^n(x-y)^{\punto k} \, \frac{1}{k!} \, \dd{^kf}{x^k}(y), \quad
%\]
%and
%\[
\hat{g}_{\ell}(z):=\hat{f}_{\ell}(z)-\sum_{k=0}^n(z-w)^k \, \frac{1}{k!} \, \frac{d^k\hat{f}_{\ell}}{dz^k}(w).
\]
%By applying point $(\mr{iii})$ of Lemma \ref{lem:holomorphic} to each $\hat{f}_{\ell}$ and 
By combining (\ref{eq:easy}), (\ref{eq:easy-2}) and Artin's theorem, we obtain that
\begin{equation} \label{eq:4}
g(x)=\sum_{\ell=0}^h\Phi_J(\hat{g}_{\ell}(z))J_{\ell}
\; \mbox{ for each }x \in B_J(y,r).
\end{equation}
%Bearing in mind the latter equality and applying point $(\mr{iii})$ of Lemma \ref{lem:holomorphic} to each $\hat{f}_{\ell}$, 
We recall the integral expression for the Taylor remainder of a holomorphic function $h$. For each $n \in \N$ and for each $z$ in a disc $B(w,r)$, it holds:
 \begin{equation}\label{eq:Taylor}
 h(z)-\sum_{k=0}^n(z-w)^k \, \frac{1}{k!} \, \frac{d^kh}{dz^k}(w)=\frac{1}{2\pi\ui} \, (z-w)^{n+1}\int_{\partial B(w,r)}h(\zeta) \, (\zeta-w)^{-n-1}(\zeta-z)^{-1} \, d\zeta.
\end{equation}
Applying \eqref{eq:Taylor} to each $\hat{f}_{\ell}$, we infer that
\begin{align}\label{eq:remainder}
g(x)&= \sum_{\ell=0}^h\Phi_J\!\!\left((2\pi\ui)^{-1}(z-w)^{n+1}\int_{\partial B(w,r)}(\zeta-z)^{-1}(\zeta-w)^{-n-1} d\zeta \, \hat{f}_{\ell}(\zeta)\right)\!J_{\ell}=\notag\\
&= \,
\sum_{\ell=0}^h\left((2\pi J)^{-1}(x-y)^{n+1}\int_{\partial B_J(y,r)}(\zeta-x)^{-1}(\zeta-y)^{-n-1} d\zeta \, \hat{f}_{\ell}(\Phi_J^{-1}(\zeta))\right)\!J_{\ell}=\notag\\
&= \,
(2\pi J)^{-1}(x-y)^{n+1}\int_{\partial B_J(y,r)}(\zeta-x)^{-1}(\zeta-y)^{-n-1} d\zeta \left(\sum_{\ell=0}^h\hat{f}_{\ell}(\Phi_J^{-1}(\zeta))J_{\ell}\right)=\notag\\
&= \,
(2\pi J)^{-1}(x-y)^{n+1}\int_{\partial B_J(y,r)}(\zeta-x)^{-1}(\zeta-y)^{-n-1} d\zeta \, f(\zeta).
\end{align}
%This shows $(\mr{iii})$. 
Let us prove $(\mr{iii})$. Define $G:B(w,r) \cap B(\overline{w},r) \lra A \otimes \C$ by setting
\[
G(z):=\int_{\partial B_J(y,r)} \Delta_{\zeta}(z)^{-1}(\zeta^c-z)(\zeta-y)^{-n-1} J^{-1} \, d\zeta \, f(\zeta).
\]
Since for each $\zeta \in \partial B_J(y,r)$ the function from $B(w,r) \cap B(\overline{w},r)$ to $A \otimes \C$, sending $z$ into $\Delta_{\zeta}(z)^{-1}(\zeta^c-z)$, is a holomorphic stem function (see \cite[Sect.~5]{GhPe_AIM}), it follows immediately that $G$ is a well--defined holomorphic stem function and hence $\mr{R}_{y,n}$ is a slice regular function. Denote by $F_1,F_2:\partial B_J(y,r) \times (B(w,r) \cap B(\overline{w},r)) \lra \C_J$ the functions such that $F_1(\zeta,z)+i F_2(\zeta,z)=\Delta_{\zeta}(z)^{-1}(\zeta^c-z)$ on $\partial B_J(y,r) \times (B(w,r) \cap B(\overline{w},r))$.

If $x=\alpha+J\beta \in B_J(y,r)$, then, bearing in mind \eqref{eq:remainder}, we have:
\begin{align*}
(2&\pi)^{-1}(x-y)^{\punto n+1} \punto \mr{R}_{y,n}(f)(x)=\\
%&= \,
%(2\pi)^{-1}(x-y)^{n+1} \left(\int_{\partial B_J(y,r)} F_1(\xi,z)(\xi-y)^{-n-1} J^{-1} \, d\xi \, f(\xi)+\right.\\
%& \,
%+\left.J\int_{\partial B_J(y,r)} F_2(\xi,z)(\xi-y)^{-n-1} J^{-1} \, d\xi \, f(\xi)\right)=\\
&= \,
(2\pi)^{-1}(x-y)^{n+1} \int_{\partial B_J(y,r)} (F_1(\zeta,z)+JF_2(\zeta,z))(\zeta-y)^{-n-1} J^{-1} \, d\zeta \, f(\zeta)=\\
&= \,
(2\pi)^{-1}(x-y)^{n+1} \int_{\partial B_J(y,r)} \Delta_{\zeta}(x)^{-1}(\zeta^c-x)(\zeta-y)^{-n-1} J^{-1} \, d\zeta \, f(\zeta)=\\
&= \,
(2\pi J)^{-1}(x-y)^{n+1} \int_{\partial B_J(y,r)} (\zeta-x)^{-1}(\zeta-y)^{-n-1} \, d\zeta \, f(\zeta)=g(x).
\end{align*}
By the representation formula for slice functions, we infer that
\begin{center}
$(2\pi)^{-1}(x-y)^{\punto n+1} \punto \mr{R}_{y,n}(f)(x)=g(x)$ for each $x \in \OO(y,r)$. 
\end{center}
If $A$ is associative and $x=\alpha+I\beta$ is an element of $\OO(y,r)$ for some $I \in \cS_A$, then we have:
\begin{align*}
\mr{R}_{y,n}(f)(x)&=
\int_{\partial B_J(y,r)} F_1(\zeta,z)(\zeta-y)^{-n-1} J^{-1} \, d\zeta \, f(\zeta)+\\
& \quad
+I\int_{\partial B_J(y,r)} F_2(\zeta,z)(\zeta-y)^{-n-1} J^{-1} \, d\zeta  \, f(\zeta)=\\
&= \,
\int_{\partial B_J(y,r)} (F_1(\zeta,z)+IF_2(\zeta,z))(\zeta-y)^{-n-1} J^{-1} \, d\zeta  \, f(\zeta)=\\
&= \,
\int_{\partial B_J(y,r)} \Delta_{\zeta}(x)^{-1}(\zeta^c-x)(\zeta-y)^{-n-1} J^{-1} \, d\zeta  \, f(\zeta).
\end{align*}

It remains to prove $(\mr{iii^{\pr }})$. Consider the element $x=\alpha+I\beta$ of $\OO(y,r)$ again and define $z_J:=\alpha+J\beta$. By the representation formula, we have that
\[
g(x)=\frac{1}{2}(g(z_J)+g(z_J^c))+\frac{I}{2}\big(J(g(z_J)-g(z_J^c))\big)
\]
and hence
\begin{align*}
\|g(x)\|_A \leq & \, %\frac{1}{2}(\|g(z_J)\|_A+\|g(z_J^c)\|_A)+\frac{C_A^2}{2}(\|g(z_J)\|_A+\|g(z_J^c)\|_A)=\\
%=& 
\frac{1+C_A^2}{2}(\|g(z_J)\|_A+\|g(z_J^c)\|_A).
\end{align*}
By using (\ref{eq:4}), \eqref{eq:Taylor} and (\ref{eq:3}), we obtain:
\begin{align*}
\|g(z_J)\|_A &\leq  \sum_{\ell=0}^h|\hat{g}_{\ell}(z)| \leq \frac{1}{r^n} \cdot \frac{|z-w|^{n+1}}{r-|z-w|} \cdot \sum_{\ell=0}^h \left({\textstyle \sup_{\partial B(w,r)}|\hat{f}_{\ell}|}\right)  \\
 &\leq 
(h+1) \, \mr{H} \left({\textstyle \sup_{\partial B_J(y,r)}\|f\|_A}\right) \cdot \frac{1}{r^n} \cdot \frac{|z-w|^{n+1}}{r-|z-w|}
\end{align*}
and similarly 
\[
\|g(z_J^c)\|_A \leq (h+1) \, \mr{H} \left({\textstyle \sup_{\partial B_J(y^c,r)}\|f\|_A}\right) \cdot \frac{1}{r^n} \cdot \frac{|z-\overline w|^{n+1}}{r-|z-\overline w|}.
\]
It follows that
\begin{align*}
\|g(x)\|_A &\leq  \frac{1+C_A^2}{2}(h+1) \, \mr{H} \left({\textstyle \sup_{\mc{F}_J(y,r)}\|f\|_A}\right) \cdot \frac{1}{r^n} \cdot \left(\frac{|z-w|^{n+1}}{r-|z-w|}+\frac{|z-\overline{w}|^{n+1}}{r-|z-\overline{w}|}\right)  \\
&\leq 
%\frac{1+C_A^2}{2}(h+1) \, \mr{H} \left({\textstyle \sup_{\partial F_J(y^c,r)}\|f\|_A}\right) \cdot \frac{1}{r^n} \cdot \left(\frac{2 \, \sigma_A(x,y)^{n+1}}{r-\sigma_A(x,y)}\right)=\\
%&=
(1+C_A^2)(h+1) \, \mr{H} \left({\textstyle \sup_{\mc{F}_J(y,r)}\|f\|_A}\right) \cdot \frac{1}{r^n} \cdot \left(\frac{\sigma_A(x,y)^{n+1}}{r-\sigma_A(x,y)}\right),
\end{align*}
%Finally, one can define $\mr{C}^{\pr}$ as $(1+C_A^2)(h+1) \, \mr{H}$, 
completing the proof of the theorem.
\end{proof}

\begin{definition}
Given a function $f:U \lra A$ defined on a non--empty open subset $U$ of $\Q_A$, we say that $f$ is \emph{$\sigma_A$--analytic} or \emph{power analytic},  if, for each $y \in U$, there exists a non--empty $\sigma_A$--ball $\Sigma$ centered at $y$ and contained in $U$, and a power series $\sum_{n \in \N}(x-y)^{\punto n}a_n$ with coefficients in $A$, which converges to $f(x)$ for each $x \in \Sigma \cap U$. 

\end{definition}

\begin{theorem}\label{poweranalytic}
Let $\OO_D$ be connected and let $f:\OO_D \lra A$ be any function. The following assertions hold.
\begin{itemize}
 \item[$(\mr{i})$] If $D \cap \R=\emptyset$, then $f$ is a slice regular function if and only if $f$ is a $\sigma_A$--analytic slice function.
 \item[$(\mr{ii})$] If $D \cap \R \neq \emptyset$, then $f$ is a slice regular function if and only if $f$ is $\sigma_A$--analytic.
\end{itemize}
\end{theorem}
\begin{proof}
From Theorem~\ref{thm:power-expansion}, if $f\in\mc{SR}(\OO_D,A)$, then $f$ is $\sigma_A$--analytic. Conversely, assume that $f$ is $\sigma_A$--analytic. If $f(x)=\sum_n(x-y)^{\punto n} a_n$ on a $\sigma_A$--ball $\Sigma$ centered at $y\in\C_J$, then $f_J(z)=f(\Phi_J(z))=\Phi_J(\sum_n(z-w)^na_n)$ for $x=\Phi_J(z)$, $y=\Phi_J(w)$,   $z$ and $w$ in an open subset of $D$. If $f$ is a slice function, the representation formula \cite[Prop.~6]{GhPe_AIM} and the smoothness result \cite[Prop.~7]{GhPe_AIM}   imply that the stem function inducing $f$, and hence $f$, are real analytic.
Moreover, from Lemma~\ref{lem:cullen} we get that $\partial f/\partial{x^c}=0$ at $y$. By the arbitrariness of the choice of $y$,  $\partial f/\partial{x^c}=0$ on the whole domain $\OO_D$.  

It remains to prove that if $D\cap\R\ne\emptyset$, the sliceness of $f$ is a consequence of its $\sigma_A$--analyticity. If $y\in\OO_D\cap\R$, then $(x-y)^{\punto n}=(x-y)^n$ and $f$ expands as $f(x)=\sum_n(x-y)^na_n$ on an euclidean  ball $B$. In particular, $f$ is slice on $B$ (cf.~Examples 2(1) of \cite{GhPe_AIM}). Let $I,J\in\cS_A$ be fixed, and consider the function $\tilde f_I:D\rightarrow A$ defined by
\[
\tilde f_I(z):=f_I(z)-\frac12(f_J(z)+f_J(\overline z))+\frac I2\left(J(f_J(z)-f_J(\overline z))\right).
\]
Since $f$ is slice on $B$, $\tilde f_I\equiv0$ on $B$. Since $f_I, f_J$ are real analytic on $D$, also $\tilde f_I\in \mscr{C}^\omega(D,A)$. Therefore $\tilde f_I\equiv0$ on $D$. Lemma~3.2 in \cite{Gh_Pe_GlobDiff} permits to conclude that $f$ is a slice function.
\end{proof}

%\noindent \textbf{Spherical expansion for slice regular functions.} 
\section{Spherical expansion for slice regular functions}\label{Spherical_expansion_for_slice_regular_functions}
We begin with two technical lemmas.

\begin{lemma} \label{lem:T}
The supremum $\Theta$, defined by setting
\[
\Theta:=\frac{1}{2\pi}\sup_{(w,r) \in \C \times \R^+}\frac{\mr{length}(\partial \Sto(w,r))}{\sqrt{r^2+|\im(w)|^2}-|\im(w)|}, 
\]
is finite. More precisely, it holds:
\begin{equation*} %\label{eq:Theta}
\Theta =\frac{1+\sqrt2}{2}\frac{\Gamma(1/4)^2}{\pi^{3/2}}< 2.85.
\end{equation*}
\end{lemma}
\begin{proof}
Define $\ell,L:\C \times \R^+ \lra \R^+$ by setting $\ell(w,r):=\mr{length}(\partial \Sto(w,r))$ and
\[
L(w,r):=\frac{\ell(w,r)}{\sqrt{r^2+|\im(w)|^2}-|\im(w)|}=\frac{\sqrt{r^2+|\im(w)|^2}+|\im(w)|}{r^2} \cdot \ell(w,r).
\]
Let $(w,r) \in \C \times \R^+$. Observe that, if $\im(w)=0$, then $L(w,r)=2\pi$. Suppose $\im(w) \neq 0$ and define $\gamma:=r/|\im(w)|>0$. It is immediate to verify that
\[
\ell(w,r)=|\im(w)| \cdot \ell(i,\gamma).
\]
It follows that $L(w,r)=L(i,\gamma)=\gamma^{-2}(1+\sqrt{1+\gamma^2})\cdot \ell(i,\gamma)$. The length of the Cassini oval of radius $\gamma$ is given by the following integral formula (cf.~\cite{MatzAMM1895}):
\[
\ell(i,\gamma)=\frac{4\gamma^2}{\sqrt{1+\gamma^2}}\int_0^{\pi/2}{\left(1-\frac{4\gamma^2}{(1+\gamma^2)^2}\sin^2\phi\right)^{-1/4}}{d\phi}\;.
\]
The normalized length $L(i,\gamma)$ is monotonically increasing from $4\pi$ for $\gamma\in(0,1)$, and monotonically decreasing to $2\pi$ for $\gamma>1$. Therefore it takes its supremum at $\gamma=1$, the radius corresponding to the Bernoulli lemniscate, which has length $\ell(i,1)=\Gamma(1/4)^2/\sqrt\pi$.

This completes the proof.
\end{proof}

\begin{lemma} \label{lem:S}
Given $w,z \in \C$, the following assertions hold.
\begin{itemize}
 \item[$(\mr{i})$] For each $n \in \N$, we have:
 \begin{equation} \label{eq:useful-1}
 |\stx_{w,n+1}(z)| \leq \sto(w,z)^n\big(\sqrt{\sto(w,z)^2+|\im(w)|^2}+|\im(w)|\big)
 \end{equation}
 and
 \begin{equation} \label{eq:useful-2}
 |\stx_{w,n+1}(z)| \geq \sto(w,z)^n\big(\sqrt{\sto(w,z)^2+|\im(w)|^2}-|\im(w)|\big).
 \end{equation}
 \item[$(\mr{ii})$] Let $n \in \N$ and let $\zeta \in \C$ with $\sto(w,\zeta)>\sto(w,z)$. Then it holds:
 \begin{equation} \label{eq:useful-3}
 |\zeta-z| \geq \frac{\big(\sto(w,\zeta)-\sto(w,z)\big)^2}{3\sto(w,\zeta)+2|\im(w)|}
 \end{equation}
 and
 \begin{equation} \label{eq:useful-4}
 |\stx_{w,n+1}(\zeta)| \geq \sto(w,\zeta)^{n+1} \cdot \frac{\sto(w,\zeta)}{\sto(w,\zeta)+2|\im(w)|}.
 \end{equation}
\end{itemize} 
\end{lemma}
\begin{proof}
Let us prove $(\mr{i})$. Define $s:=\sto(w,z)$. Fix $n \in \N$. If $n=2m+1$ for some $m \in \N$, then $|\stx_{w,n+1}(z)|=s^{n+1}=s^ns$. %and hence (\ref{eq:useful-2-1}) holds.
On the other hand, we have that
\[
\sqrt{s^2+|\im(w)|^2}-|\im(w)| \leq s \leq \sqrt{s^2+|\im(w)|^2}+|\im(w)|
\]
and hence (\ref{eq:useful-1}) and (\ref{eq:useful-2}) hold. Suppose now that $n=2m$ for some $m \in \N$. Observe that $|\stx_{w,n+1}(z)|=s^n|z-w|$. In this way, (\ref{eq:useful-1}) and (\ref{eq:useful-2}) follow immediately from (\ref{eq:sto-inequality}).

It remains to show $(\mr{ii})$. Define $r:=\sto(w,\zeta)$. Since $\sto$ is a pseudo--metric on $\C$, we know that
\begin{equation} \label{eq:useful-5}
\sto(\zeta,z) \geq r-s.
\end{equation}
On the other hand, since $r^2=|w-\zeta||w-\overline\zeta|$, it is immediate to verify that
\begin{equation} \label{eq:useful-6}
|\im(\zeta)| \leq \sqrt{r^2+|\im(w)|^2}.
\end{equation}
By combining inequalities (\ref{eq:useful-5}), (\ref{eq:useful-6}) with the first inequality of (\ref{eq:sto-inequality}), we obtain (\ref{eq:useful-3}):
\begin{align*}
|\zeta-z| \geq&
\sqrt{\sto(\zeta,z)^2+|\im(\zeta)|^2}-|\im(\zeta)| \geq \sqrt{(r-s)^2+|\im(\zeta)|^2}-|\im(\zeta)|= \\
= &
\frac{(r-s)^2}{\sqrt{(r-s)^2+|\im(\zeta)|^2}+|\im(\zeta)|} \geq \frac{(r-s)^2}{r+2|\im(\zeta)|} \geq \frac{(r-s)^2}{3r+2|\im(w)|}\;.
\end{align*}
Finally, applying (\ref{eq:useful-2}) with $z=\zeta$, %and the inequality $r \geq |\im(w)|$ again, 
we obtain (\ref{eq:useful-4}), as desired:
\begin{align*}
|\stx_{w,n+1}(\zeta)|=&
r^n|w-\zeta| \geq r^n\big(\sqrt{r^2+|\im(w)|^2}-|\im(w)|\big) =\\
=&
r^n \cdot \frac{r^2}{\sqrt{r^2+|\im(w)|^2}+|\im(w)|} \geq r^{n+1} \cdot \frac{r}{r+2|\im(w)|}\;.
\end{align*}
\end{proof}

In the next result, we perform  on a holomorphic function the expansion procedure described in \cite{StoppatoAdvMath2012} (see Sections 2, 3 and 4 of \cite{StoppatoAdvMath2012}). Observe that we do not require that the $\sto_A$--ball $\Sto(w,r)$ with boundary the Cassini oval  where we expand is connected: the radius $r$ can also be smaller or equal to $|\im(w)|$. In this case, $\Sto(w,r)$ does not intersect the real axis.

%\footnote{Bisognerebbe specificare cosa c'e' di nuovo qui rispetto a Caterina: il $\Theta$, la stima in $(\mr{iii})$ e soprattutto il fatto che la condizione $r>|\im(w)|$ non \`e richiesta.}

\begin{lemma} \label{lem:sto-holomorphic}
Let $E$ be a non--empty bounded open subset of $\C$, let $g:E \lra \C$ be a holo\-morphic function, let $w \in E$ and let $r \in \R^+$ such that $\cl(\Sto(w,r)) \subset E$. Define $G:=\sup_{\partial \Sto(w,r)}|g|$. %Suppose that $g$ admits a continuous extension on the closure of $E$ in $\C$.
Then the following properties are satisfied.
\begin{itemize}
 \item[$(\mr{i})$] For each $n \in \N$, define $A_n \in \C$ by setting
 \begin{equation}\label{An}
A_n:=\frac{1}{2\pi\ui}\int_{\partial \Sto(w,r)}\frac{g(\zeta)}{\stx_{w,n+1}(\zeta)} \, d\zeta.
 \end{equation}
 Then $g$ expands as follows:
 \[
 g(z)=\sum_{n \in \N}\stx_{w,n}(z) \, A_n \; \mbox{ for each }z \in \Sto(w,r).
 \]
 \item[$(\mr{ii})$] For each $n \in \N$, it holds:
 \[
 \left|A_n\right| \leq \frac{\Theta G}{r^n}.
 \]
 \item[$(\mr{iii})$] For each $n \in \N$ and for each $z \in \Sto(w,r)$, it holds:
 \begin{equation} \label{eq:iii-1}
 g(z)-\sum_{k=0}^n\stx_{w,k}(z) \, A_k=\frac{1}{2\pi\ui} \, \stx_{w,n+1}(z)\int_{\partial \Sto(w,r)}\frac{g(\zeta)}{(\zeta-z)\stx_{w,n+1}(\zeta)} \, d\zeta
 \end{equation}
 and
 \[
 \left| \, g(z)-\sum_{k=0}^n\stx_{w,k}(z) \, A_k \right| \leq \Theta G \left(\frac{\sto(w,z)}{r}\right)^n  \frac{(3r+2|\im(w)|)(r+2|\im(w)|)^2}{r(r-\sto(w,z))^2}. %\left({\textstyle \sup_{\partial B(w,r)}|g|}\right) \cdot \frac{1}{r^n} \cdot \frac{|z-w|^{n+1}}{r-|z-w|}.
 \]
 %where $\delta_n(z):=3r$ if $n$ is even and $\delta_n(z):=\sto(w,z)$ otherwise.
\end{itemize}
\end{lemma}
\begin{proof}
Let $n \in \N$. Thanks to (\ref{eq:useful-2}) and Lemma \ref{lem:T}, we have:
\[
|A_n| \leq \frac{1}{2\pi}\int_{\partial\Sto(w,r)}\frac{|g(\zeta)|}{|\stx_{w,n+1}(\zeta)|} \, d|\zeta| \leq \frac{G}{r^n} \left(\frac{1}{2\pi} \cdot \frac{\mr{length}(\partial \Sto(w,r))}{\sqrt{r^2+|\im(w)|^2}-|\im(w)|}\right) \leq \frac{\Theta G}{r^n}.
\]
This proves $(\mr{ii})$. Moreover, it follows that $\limsup_{n \rightarrow +\infty}|A_n|^{1/n} \leq 1/r$. In this way, by applying point $(\mr{ii})$ of Theorem \ref{thm:radii} with $A=\C$ and $\| \cdot \|_A$ equal to the usual euclidean norm $|\cdot|$ of $\C$, we obtain that the series $\sum_{n \in \N}\stx_{w,n}(z) \, A_n$ converges on $\Sto(w,r)$.

Let $\zeta \in \partial \Sto(w,r)$ and let $z \in \Sto(w,r)$. Observe that it holds:
\begin{equation*} %\label{eq:w}
\frac{1}{\zeta-z}=\frac{1}{\zeta-w}+\frac{z-w}{\zeta-w} \cdot \frac{1}{\zeta-z}
\; \; \; \mbox{ and } \; \; \;
\frac{1}{\zeta-z}=\frac{1}{\zeta-\overline{w}}+\frac{z-\overline{w}}{\zeta-\overline{w}} \cdot \frac{1}{\zeta-z}.
\end{equation*}
By using alternatively the preceding two equalities, we obtain:
\begin{equation} \label{eq:cauchy-kernel}
\frac{1}{\zeta-z}=\sum_{k=0}^n\frac{\stx_{w,k}(z)}{\stx_{w,k+1}(\zeta)}+\frac{\stx_{w,n+1}(z)}{\stx_{w,n+1}(\zeta)} \cdot \frac{1}{\zeta-z}
\; \; \mbox{ for each }n \in \N.
\end{equation}
By combining the classical Cauchy formula with (\ref{eq:cauchy-kernel}), we infer that
\begin{align*}
g(z)= & \, \frac{1}{2\pi\ui}\int_{\partial\Sto(w,r)}\frac{g(\zeta)}{\zeta-z} \, d\zeta=\\
=& \,
\sum_{k=0}^n\stx_{w,k}(z) \, A_k+\frac{1}{2\pi\ui} \, \stx_{w,n+1}(z)\int_{\partial \Sto(w,r)}\frac{g(\zeta)}{(\zeta-z)\stx_{w,n+1}(\zeta)} \, d\zeta,
\end{align*}
which is equivalent to (\ref{eq:iii-1}). On the other hand, Lemmas~\ref{lem:T} and \ref{lem:S} %points (\ref{eq:useful-1}), (\ref{eq:useful-3}) and (\ref{eq:useful-4}) 
imply that
\begin{align*}
&\left| \, g(z)-\sum_{k=0}^n\stx_{w,k}(z) \, A_k \right| \leq 
\frac{1}{2\pi} \, |\stx_{w,n+1}(z)|\int_{\partial \Sto(w,r)}\frac{|g(\zeta)|}{|\zeta-z|  |\stx_{w,n+1}(\zeta)|} \, d|\zeta| \leq\\
& \leq
\frac{\sto(w,z)^n\big(\sto(w,z)+2|\im(w)|\big)}{2\pi} \cdot G \cdot \frac{3r+2|\im(w)|}{(r-\sto(w,z))^2} \cdot \frac{r+2|\im(w)|}{r^{n+2}} \cdot \mr{length}(\partial\Sto(w,r))\leq\\
&
\leq G\; \frac{\sto(w,z)^n}{r^n} \cdot \frac{(3r+2|\im(w)|)(r+2|\im(w)|)^2}{r(r-\sto(w,z))^2} \cdot 
\left(\frac{\mr{length}(\partial \Sto(w,r))}{2\pi r}\right) \leq \\
&
\leq G\; \frac{\sto(w,z)^n}{r^n}  \cdot \frac{(3r+2|\im(w)|)(r+2|\im(w)|)^2}{r(r-\sto(w,z))^2} \cdot \left( \frac{\mr{length}(\partial \Sto(w,r))}{2\pi\sqrt{r^2+|\im(w)|^2}-|\im(w)|}\right) \leq\\
&
\leq \Theta G\; \frac{\sto(w,z)^n}{r^n}\;  \frac{(3r+2|\im(w)|)(r+2|\im(w)|)^2}{r(r-\sto(w,z))^2}\;.
\end{align*}
This proves point $(\mr{iii})$ and ensures that the series $\sum_{n \in \N}\stx_{w,n}(z) \, A_n$ converges to $g(z)$ for each $z \in \Sto(w,r)$, because $\sto(w,z)<r$. Point $(\mr{i})$ is proved and the proof is complete.
\end{proof}

%We have: 

\begin{theorem} \label{thm:spherical-expansion} 
Let $f \in \mc{SR}(\OO_D,A)$, let $y \in \OO_D$, let $J \in \cS_A$ such that $y \in \C_J$ and let $r \in \R^+$ such that $\cl(\Sto_A(y,r)) \subset \OO_D$. Define $\Sto_J(y,r):=\Sto_A(y,r) \cap \C_J$ and $S_f:=\sup_{\partial \Sto_J(y,r)}\|f\|_A$. 
Then the following assertions hold.
\begin{itemize}
 \item[$(\mr{i})$] There exists a unique sequence $\{s_n\}_{n \in \N}$ in $A$ such that $f$ expands as follows:
 \[
 f(x)=\sum_{n \in \N}\stx_{y,n}(x)s_n \; \mbox{ for each }x \in 
 \Sto_A(y,r).
 \]
 We call $s_n$ the \emph{$\mr{n}^{\mr{th}}$-spherical number of $f$ at $y$}.
 \item[$(\mr{ii})$] For each $n \in \N$, it holds:
 \begin{equation}\label{sn}
 s_n=(2\pi J)^{-1}\int_{\partial \Sto_J(y,r)}(\stx_{y,n+1}(\zeta))^{-1} \, d\zeta \, f(\zeta).
 \end{equation}
 Furthermore, there exists a positive real constant $C$, depending only on $\mk{a}$ and on $\| \cdot \|_A$, such that
  \begin{equation}\label{estimate_sn}
 \|s_n\|_A \leq \frac{C S_f}{r^n}  \; \mbox{ for each }n \in \N.
  \end{equation}
% \item[$(\mr{ii}^{\pr})$] For each $n \in \N$, define $E_n \in A$ by setting
% \begin{equation*}
% E_n:=
%\left\{
% \begin{array}{ll}
%  \frac{1}{m!} \, (2\im(y))^m \frac{\partial^m}{\partial x^m}f(y) & \mbox{ if } \; n=2m \vspace{.3em}\\
%  \frac{1}{m!} \, (-2\im(y))^m \frac{\partial^m}{\partial x^m}f(y^c) & \mbox{ if } \; n=2m+1.
% \end{array}
%\right.
% \end{equation*}
%Denote by $\mk{E}$ and $\mk{s}$ the infinite vectors $(E_n)_{n \in \N}$ and $\big((-2\im(y))^ns_n\big)_{n \in \N}$ in $A^{\N}$, respectively. Then $\mk{s}$ is the unique solution of the following infinite lower triangular linear system:
% \begin{equation*}
% \mk{e} \cdot \mk{s}=\mk{E}.
% \end{equation*}
% In other words, for each $n \in \N$, it holds:
% \begin{equation*}
% s_n=(-2\im(y))^{-n}\det(e_n|\mk{E}_n),
% \end{equation*}
%where $\mk{E}_n$ is the column vector $(E_0,E_1,\ldots,E_n)$.
 \item[$(\mr{iii})$] Let $n$ be an arbitrary integer in $\N$. It holds:
 \begin{equation}\label{equalityiii}
 f(x)-\sum_{k=0}^n\stx_{y,k}(x)s_k=(2\pi)^{-1} \, \left(\stx_{y,n+1} \punto \mk{R}_{y,n}(f)\right)(x)
 \end{equation}
 for each $x \in \Sto_A(y,r)$, where $\mk{R}_{y,n}(f):\Sto_A(y,r) \lra A$ is the slice regular function induced by the holomorphic stem function
% \[
% \Sto(w,r) \ni z \longmapsto \int_{\partial \Sto_J(y,r)} \Delta_{\xi}(\zeta)(\xi^c-\zeta)J^{-1}(\stx_{y,n+1}(\xi))^{-1} \, d\xi \, f(\xi) \in A \otimes \C,
% \]
%\marginpar{$\zeta$ diventa $z$ e $\xi$ diventa $\zeta$?}
% where $w=\theta+\ui\eta \in D$ if $y=\theta+J\eta$ and $\zeta=\alpha+\UI\beta \in \R \otimes \C$ if $z=\alpha+\ui\beta$.
 \[
 \Sto(w,r) \ni z \longmapsto \int_{\partial \Sto_J(y,r)} \Delta_{\zeta}(z)^{-1}(\zeta^c-z)J^{-1}(\stx_{y,n+1}(\zeta))^{-1} \, d\zeta \, f(\zeta),
 \]
where $w=\xi+\ui\eta \in D$ if $y=\xi+J\eta$.  
%In particular, if $x \in \Sto_J(y,r)$, then
% \[
% f(x)-\sum_{k=0}^n\stx_{y,k}(x)s_k=(2\pi J)^{-1} \, \stx_{y,n+1}(x)\int_{\partial \Sto_J(y,r)} (\zeta-x)^{-1}(\stx_{y,n+1}(\zeta))^{-1} d\zeta \, f(\zeta).
% \] 
 Furthermore, if $A$ is associative, then we have
 \[
 \mk{R}_{y,n}(f)(x)=\int_{\partial \Sto_J(y,r)} \Delta_{\zeta}(x)^{-1}(\zeta^c-x)J^{-1}(\stx_{y,n+1}(\zeta))^{-1} \, d\zeta \, f(\zeta)
 \]
 for each $x \in \Sto_A(y,r)$.
 \item[$(\mr{iii^{\pr}})$] There exists a positive real constant $C^{\pr}$, depending only on $\mk{a}$ and on $\| \cdot \|_A$, such that
 \[
 \left\|  f(x)-\sum_{k=0}^n\stx_{y,k}(x)s_k\right\|_A \leq C^{\pr} \, S_f \frac{\sto_A(x,y)^n}{r^n}\, \frac{(3r+2|\im(w)|)(r+2|\im(w)|)^2}{r(r-\sto_A(x,y))^2}
 \]
 for each $n \in \N$ and for each $x \in \Sto_A(y,r)$.
\end{itemize}
\end{theorem}
\begin{proof}
We can proceed as in the proof of  Theorem~\ref{thm:power-expansion}. Applying Lemma~\ref{lem:sto-holomorphic} to each holomorphic $\hat f_\ell$ of decomposition \eqref{eq:easy}, we get the spherical expansion at each $x\in \Sto_J(y,r)$:
\[
f(x)=f_J(z)=\sum_{\ell=0}^h\Phi_J\left(\sum_{n\in\N}\stx_{w,n}(z)A_{n,\ell}\right)J_\ell=\sum_{n\in\N}\stx_{y,n}(x)s_n,
\]
where $A_{n,\ell}$ are the complex numbers given by formula \eqref{An} and $s_n=\sum_\ell \Phi_J(A_{n,\ell})J_\ell\in A$.
From $\sum_\ell(\Phi_J\circ\hat f_\ell)J_\ell=f_J$, we get formula \eqref{sn} for $s_n$. Estimate \eqref{estimate_sn} follows from Lemma~\ref{lem:sto-holomorphic}(ii), with constant $C=C_A(h+1)\Theta$. The same estimate gives also the uniqueness of the expansion coefficients.

Let us prove (iii). From equation \eqref{eq:iii-1} of Lemma~\ref{lem:sto-holomorphic} applied to each component $\hat f_\ell$, we get that if $x \in \Sto_J(y,r)$, then
 \[
 f(x)-\sum_{k=0}^n\stx_{y,k}(x)s_k=(2\pi J)^{-1} \, \stx_{y,n+1}(x)\int_{\partial \Sto_J(y,r)} (\zeta-x)^{-1}(\stx_{y,n+1}(\zeta))^{-1} d\zeta \, f(\zeta).
 \] 
On the other hand, for each $x \in \Sto_J(y,r)$, it holds
\begin{align*}
(2&\pi)^{-1} \, \left(\stx_{y,n+1} \punto \mk{R}_{y,n}(f)\right)(x)=\\
&=(2\pi)^{-1} \, \stx_{y,n+1}(x) \int_{\partial \Sto_J(y,r)} \Delta_{\zeta}(z)^{-1}(\zeta^c-z)J^{-1}(\stx_{y,n+1}(\zeta))^{-1} \, d\zeta \, f(\zeta)=\\
&=(2\pi J)^{-1} \, \stx_{y,n+1}(x) \int_{\partial \Sto_J(y,r)}(\zeta-z)(\stx_{y,n+1}(\zeta))^{-1} \, d\zeta \, f(\zeta).
\end{align*}
The representation formula for slice functions implies that \eqref{equalityiii} holds on the whole circular domain $\Sto(y,r)$.
The estimate $(\mr{iii^{\pr}})$ of the remainder follows in a similar way from the estimate of Lemma~\ref{lem:sto-holomorphic} and the representation formula again.
\end{proof}

\begin{remark}
In view of parts (i) and (ii) of Theorem~\ref{thm:spherical-expansion}, the spherical numbers of $f$ at $y$ can be expressed as solutions of the infinite linear system introduced in Theorem~\ref{thm:coefficients}. In particular, the spherical number $s_1$ is equal to the spherical derivative $\partial_s f(y)$ defined in \cite[Def.~6]{GhPe_AIM}.
\end{remark}

\begin{example}
Let $A=\q$ and let $J\in\cS_\q$ be fixed. Consider the slice regular function $f$ on $\q\setminus\R$ defined, for each $x=\alpha+\beta I_x$, with $\beta>0$, by
\[f(x)=1-I_xJ.\]
The function $f$ is induced by the locally constant stem function taking values $2$ on $\C^+$ and $0$ on $\C^-$. We compute the power and spherical expansions of $f$ at $y=J$. Since the Cullen (or slice) derivatives $\partial_C^nf(y)$ vanish  for each $n>0$, the power expansion $\mr{P}(x)$ reduces to the first term: $\mr{P}(x)=2$. In this case, the maximal $\sigma_\q$--ball centered at $J$ on which the power expansion converges to $f$ is the domain $\Sigma_\q(J,1)=\{q\in\C_J\;|\;|q-J|<1\}$, which has empty interior w.r.t.\ the eucledean topology of $\q$. On the other hand, the spherical expansion converges to $f$ on a non-empty open domain of $\q$. By solving the system \eqref{eq:system}, one obtains the spherical numbers of $f$ at $J$:
\[
s_0=2,\quad s_n=\begin{cases}
2(-1)^k\binom{2k}k(2J)^{-n}=4^{-k}\binom{2k}k(-J)&\text{\quad if $n=2k+1$ is odd,}\\
(-1)^k\binom {2k}k(2J)^{-n}=4^{-k}\binom{2k}k&\text{\quad if $n=2k>0$ is even.}\end{cases}
\]
Therefore the spherical expansion of $f$ at $J$ has the following form:
\begin{align*}
\mr{S}(x)&=2+(x-J)(-J)+\Delta_J(x)\frac12+\Delta_J(x)(x-J)\left(-\frac J2\right)+\\
&\ \qquad+(\Delta_J(x))^2\frac38+(\Delta_J(x))^2(x-J)\left(-\frac38J\right)+\cdots=\\
&=1-\sum_{k=0}^{+\infty} \frac1{4^k}\binom{2k}k (1+x^2)^kxJ.
\end{align*}
As a consequence of Theorem~\ref{thm:spherical-expansion} applied on the $\sto_\q$--balls $\Sto_\q(J,r)$ of radius $r<1$, this series converges uniformly to $f$ on the compact subsets of the $\sto_\q$--ball $\Sto_\q(J,1)$.

\end{example}

\begin{definition}
Given a function $f:V \lra A$ from a non--empty circular open subset $V$ of $\Q_A$ into $A$, we say that $f$ is \emph{$\sto_A$--analytic} or \emph{spherically analytic},
 if, for each $y \in V$, there exists a non--empty $\sto_A$--ball $\Sto$ centered at $y$ and contained in $V$, and a spherical series $\sum_{n \in \N}\stx_{y,n}(x)s_n$ with coefficients in $A$, which converges to $f(x)$ for each $x \in \Sto \cap V$.
\end{definition}

\begin{theorem}
Let $\OO_D$ be connected and let $f:\OO_D \lra A$ be any function. The following assertions hold.
\begin{itemize}
 \item[$(\mr{i})$] If $D \cap \R=\emptyset$, then $f$ is a slice regular function if and only if $f$ is slice and spherically analytic.
 \item[$(\mr{ii})$] If $D \cap \R \neq \emptyset$, then $f$ is a slice regular function if and only if $f$ is spherically analytic.
\end{itemize}
\end{theorem}
\begin{proof}
From Theorem~\ref{thm:spherical-expansion}, if $f\in\mc{SR}(\OO_D,A)$, then $f$ is spherically analytic. Conversely, if $f$ is spherically analytic and $f(x)=\sum_n\stx_{y,n}(x)s_n$ on a $\sto_A$--ball $\Sto$ centered at $y\in\C_J$, then $f_J(z)=f(\Phi_J(z))=\Phi_J(\sum_n\stx_{y,n}(z)s_n)$ for $x=\Phi_J(z)$, $y=\Phi_J(w)$,   $z$ and $w$ in an open subset of $D$. If $f$ is a slice function, the representation formula and the smoothness result \cite[Propositions~6 and 7]{GhPe_AIM}   imply that %the stem function inducing $f$, and hence $f$, are 
$f$ is real analytic.
Moreover, from Lemma~\ref{lem:cullen} we get that $\partial f/\partial{x^c}=0$ at $y$. By the arbitrariness of the choice of $y$,  $\partial f/\partial{x^c}=0$ on the whole domain $\OO_D$.  

It remains to prove that if $D\cap\R\ne\emptyset$, the sliceness of $f$ is a consequence of its spherical analyticity. If $y\in\OO_D\cap\R$, then $\stx_{y,n}(x)=(x-y)^n$ and $f$ expands as $f(x)=\sum_n(x-y)^ns_n$ on an euclidean  ball $B$.  Now we can conclude as in the last part of the proof of Theorem~\ref{poweranalytic}.

\end{proof}

%%%%%%%%%%%%%%%%%%%%%

%\bibliographystyle{abbrv}
%\bibliography{Ref_AP}{}

\end{document}